\documentclass{article}
\usepackage{graphicx} 
\usepackage[utf8]{inputenc}
\usepackage[affil-it]{authblk}
\usepackage{graphicx} 
\usepackage{subcaption}
\usepackage{comment,amsmath,amsbsy,amssymb,graphicx,dsfont,upgreek,textcomp,braket,setspace,blindtext,hyperref,verbatim,amsthm,mathrsfs,mathtools,graphicx,float,enumerate,cleveref,cases,bigints}
\usepackage[font={scriptsize,it}]{caption}
\usepackage[margin=1in]{geometry}

\hypersetup{
    colorlinks=true, 
    linktoc=all,     
    linkcolor=blue,  
}
\usepackage{tocloft} 
\usepackage[table,svgnames]{xcolor}
\usepackage{tikz}
\usetikzlibrary{arrows}
\usepackage[toc,page]{appendix}
\usepackage{xcolor}

\Crefname{appsec}{appendix}{appendices}
\numberwithin{equation}{section}

\newtheorem{theorem}{Theorem}[section]
\newtheorem{lemma}{Lemma}[section]
\newtheorem{corollary}{Corollary}[section]
\newtheorem{remark}{Remark}[section]

\newtheorem{proposition}{Proposition}[section]

\title{\Large Dynamics of periodic fractional discrete NLS in the continuum limit\vspace{-1ex}}
\author{\large Brian Choi \vspace{-2ex}\thanks{University of Tennessee at Chattanooga, \texttt{ choigh@bu.edu}\\
\textbf{Keywords.} NLS, Fractional calculus, Continuum limit, Modulational instability, Dispersive estimate\\
\textbf{MSC 2020} 35Q55, 81T27, 65M06, 34A34, 37K60}}

\date{\normalsize October, 2025\vspace{-1ex}}
\begin{document}
\maketitle\vspace{-5ex}
\maketitle
\begin{abstract}
The fractional discrete nonlinear Schr\"odinger equation (fDNLS) is studied on a periodic lattice from the analytic and dynamic perspective by varying the mesh size $h>0$ and the nonlocal L\'evy index $\alpha \in (0,2]$. We show that the discrete system converges to the fractional NLS as $h \rightarrow 0$ below the energy space by directly estimating the difference between the discrete and continuum solutions in $L^2(\mathbb{T})$ using the discrete periodic Strichartz estimates. The sharp convergence rate via the finite difference method (FDM) is shown to be $O(h^{\frac{\alpha}{2+\alpha}})$ in the energy space. To further illustrate the convergent behavior of fDNLS, we survey various dynamical behaviors of the continuous wave (CW) solutions in the context of modulational instability, emphasizing the interplay between linear dispersion (or lattice diffraction), characterized by the nonlocal lattice coupling, and nonlinearity. In particular, the transition as $h \rightarrow 0$ from the \textit{linear} dependence of maximum gain $\Omega_m$ on the amplitude $A$ of CW solutions to the \textit{quadratic} dependence is shown analytically and numerically.    
\end{abstract}
\smallskip



\section{Introduction.}

In this paper, the fractional discrete nonlinear Schr\"odinger equation (fDNLS)
\begin{equation}\label{fdnls}
\begin{split}
i \dot{u}_h &= (-\Delta_h)^{\frac{\alpha}{2}} u_h + \mu |u_h|^2 u_h,\ (x,t) \in \mathbb{T}_h \times \mathbb{R},\\
u_h(x,0) &= u_{h,0}(x)
\end{split}
\end{equation}
on a periodic lattice is studied featuring the continuum limit at low regularity and the modulational instability of CW solutions governed by nonlocal long-range interactions described by the L\'evy index $\alpha \in (1,2)$. The formal continuum limit of \eqref{fdnls} as $h \rightarrow 0$ yields the fractional nonlinear Schr\"odinger equation (fNLS) \eqref{fnls} where the model is defocusing/focusing for $\mu = \pm 1$, respectively. It is immediately observed that \eqref{fdnls} is a FDM model of \eqref{fnls} where the time variable is not discretized. For notation, see \Cref{notation}.

For $\alpha = 2$, \eqref{fnls} recovers the well-studied NLS whose well-posedness theory with the periodic boundary condition goes back to \cite{Bourgain1993}. The method used in this reference based on the Bourgain space motivated the rigorous study of nonlocal \eqref{fnls} by \cite{Yonggeun,3be5559080074497bb318d205d3439db} where the local well-posedness in $H^s(\mathbb{T})$ for $s \geq s_{fNLS} := \frac{2-\alpha}{4}$ was shown. In the non-compact Euclidean space, the well-posedness theory using the Strichartz estimates was shown in \cite{1534-0392_2015_6_2265,dinh:hal-01426761}. Motivated from nonlinear optics, the mixed-fractional variant of fNLS on $\mathbb{R}^2$ was studied in \cite{choi2022well} where the coupling strengths of the nonlocal interaction were assumed to be non-homogeneous in the two transverse directions with respect to the propagation axis.

The NLS is a ubiquitous model in nonlinear wave phenomena that arises as the homogenized equation in various physical applications including the pulse propagation of intense laser beam in nonlinear media and the Bose-Einstein condensates, or the collective behavior of bosons in an ultra-cold temperature, via the Gross-Pitaevskii hierarchy. A recent generalization of NLS, the fractional NLS, introduces nonlocality as a parameter that measures strong correlations between distant lattice sites. One of the motivations for studying fNLS comes from fractional quantum mechanics \cite{laskin2000fractional} where the Feynman path integral formalism based on Brownian-like paths was extended to the $\alpha$-stable L\'evy-like paths.
Meanwhile, our interest extends to the relationship between fNLS and its discrete analog. The long-range variant of DNLS is relevant not only in numerical analysis but also in physical phenomena that are inherently discrete. In a fixed anharmonic lattice, soliton dynamics with the coupling strengths decaying algebraically, as opposed to the nearest-neighbor interaction, was studied in \cite{mingaleev1998solitons,PhysRevE.55.6141}. When $\alpha = 2$, DNLS, among many others, describes pulse propagation in discrete waveguide arrays. DNLS is a well-established model where various results, both theoretical and numerical, can be found in a monograph \cite{kevrekidis2009discrete}. Another important feature of discreteness is the Peierls-Nabarro barrier, studied in \cite{KivCam} applied to DNLS, where the lattice structure yields an effective energy barrier that eventually pins the transport of a pulse. For an extension of this work to a nonlocal setting, see \cite{choi2023localization}.

One of the first rigorous (weak) convergence results of fDNLS to fNLS on $h\mathbb{Z}$ as $h \rightarrow 0$ under a general interaction kernel was shown in \cite{kirpatrick}, to be strengthened to strong convergence \cite{hong2019strong} in $L^2(\mathbb{R})$ under certain hypotheses when $\alpha \in (0,2) \setminus \{1\}$. The strong convergence in $L^2(\mathbb{R}^2)$ was shown in \cite{choi2023continuum} for energy-subcritical data corresponding to $\alpha \in (1,2)$. Our current work contrasts with those of Ignat and Zuazua \cite{ignat2005two,ignat2009numerical,ignat2012convergence}, which are based on preconditioning the numerical scheme, via the Fourier filtering or the two-grid algorithm, that avoids the effect of weak dispersion whose weaker dispersive decay properties were studied in \cite{stefanov2005asymptotic}; however we derive uniform blow-up of discrete dispersive estimates in \Cref{blow-up} by modifying the method in \cite{ignat2009numerical}. Instead our approach does not modify the finite-difference scheme, and therefore, the weak dispersive effects rising from the degenerate phase of the discrete Laplacian need to be addressed. Note that this degeneracy is a purely discrete phenomenon, which leads to a derivative loss in the Strichartz estimates (see \Cref{strichartz}).

A further motivation to introduce nonlocal operators stems from the convergent mean-field behavior of lattice dynamics under long-range interactions to a homogenized nonlocal partial differential equation on a smooth domain. While the analysis on the continuum limit for fDNLS in $\mathbb{R}$ has been studied in \cite{hong2019strong}, an analogous study in $\mathbb{T}$ is absent in the literature, and it is our intention to fill this gap. Naturally, our method builds upon that developed in \cite{hong2021finite} that studied the $h \rightarrow 0$ problem of local DNLS in $\mathbb{T}^2$, where the role of nonlocality is emphasized in our article. One of the main questions that we raise is, is it possible to establish the rigorous convergence as $h \rightarrow 0$ in the mathematical framework of \cite{hong2021finite,hong2019strong} whenever the formal limit, the periodic fNLS, is locally well-posed, i.e., for $s \geq s_{fNLS}$? Our main result provides a partial progress that proves the convergence of FDM for initial data of low regularity.
\begin{theorem}\label{result}
    Let $\alpha \in (1,2]$. For any $s > s_0 := \max(\frac{3-\alpha}{4},\frac{1}{3})$ and $u_0 \in H^s(\mathbb{T})$, let $S(t)u_0$ and $S_h(t) d_h u_0$ denote the well-posed solutions constructed in \Cref{lwp_fnls} and \Cref{lwp_fdnls}, respectively. Then there exists $C(\| u_0 \|_{H^s},\alpha) > 0$ such that the error estimate
    \begin{equation}\label{continuum_local}
        \| p_h S_h(t) d_h u_0 - S(t) u_0\|_{L^2(\mathbb{T})} \leq C h^{\frac{2s}{2+\alpha}},
    \end{equation}
    holds for all $t \in [0,T]$ where $T = T(\| u_0 \|_{H^s},\alpha)>0$. If $s = \frac{\alpha}{2}$ (energy space), then $T>0$ can be arbitrarily large and the order of convergence $\frac{\alpha}{2+\alpha}$ is sharp.
\end{theorem}

The class of solutions considered above strictly contains Sobolev functions of continuity and finite-energy functions since $s_0(\alpha) < \frac{1}{2} < \frac{\alpha}{2}$. However \Cref{result} is not sharp as $s_{fNLS} < s_0(\alpha)$ and thus more investigation is needed here. This non-sharpness rises in the approximation of an oscillatory sum via integral (see \Cref{lemma:zygmund}), which is used to obtain our main tool, the uniform short-time dispersive estimates \eqref{nonsharp_strichartz}. Furthermore the sharpness of the order of convergence, $O(h^{\frac{\alpha}{2+\alpha}})$ in the energy space, is a new result, and therefore, to minimize the error in discretization (the LHS of \eqref{continuum_local}), it is not only sufficient but also \textit{necessary} that $h$ be small. See also a related low-regularity analysis with damping and forcing \cite{zhang2021long}.

On the other hand, our interest extends to the dynamical behavior of special solutions, where the family of continuous wave (CW) solutions is considered for concreteness, in the continuum limit. To this end, fractional modulational instability (MI) is treated analytically and numerically with nonlocality and discreteness as parameters. Localization of nonlinear waves where a breather-like excitation rises due to a small perturbation in its spectrum has been an active area of research, including the Stokes wave \cite{zakharov2006freak}, MI in soliton dynamics \cite{gelash2019bound}, and mixed-fractional NLS and fNLS \cite{zhang2017modulational,Copeland:20,duo2021dynamics} just to name a few. The MI gain spectrum, maximum gain, and the corresponding fastest-growth frequencies are explicitly computed. 

In \Cref{notation}, mathematical background and notation are introduced. In \Cref{Strichartz_section}, the dispersive estimates for \eqref{fdnls} are developed, followed by the proof of \Cref{result} in \Cref{proof_thm1}. In \Cref{MI_section}, theoretical and numerical studies on fractional MI are presented that emphasize the role of discreteness and the L\'evy index. Some technical results are given in \Cref{Uniform}.

\section{Mathematical background.}\label{notation}

Let $h = \frac{\pi}{M}$ where $M \geq 1$ is an integer. The periodic lattice of uniform mesh is defined as 
\begin{equation*}
\mathbb{T}_h = \{x = hj: j = -M, \dots , M-1\} \cong \mathbb{Z} / (2M \mathbb{Z}).     
\end{equation*}
The discrete Lebesgue space is defined with the normalized counting measure $d\mu_h$ given by
\begin{equation*}
    L^1_h := L^1((\mathbb{T}_h, d\mu_h);\mathbb{C})\ni f \mapsto \int_{\mathbb{T}_h} f d\mu_h := h \sum_{x \in \mathbb{T}_h} f(x),
\end{equation*}
where the family of discrete Lebesgue spaces $L^p_h$ is defined similarly for $p \in [1,\infty]$. The dual space $\mathbb{T}_h^*$ is defined as $Hom(\mathbb{T}_h,S^1) \cong \{-M,\dots,M-1\} \cong \mathbb{Z} / (2M \mathbb{Z})$ where each $k \in \mathbb{T}_h^{*}$ acts on $x \in \mathbb{T}_h$ by $(k,x)\mapsto e^{i k x}$. The Plancherel's Thoerem gives that the spaces of $L^2$ functions on the lattice and its dual are isomorphic under the discrete (inverse) Fourier transform defined by
\begin{equation*}
\mathcal{F}_h [f] (k) = h\sum_{x \in \mathbb{T}_h} f(x) e^{-ikx},\ \mathcal{F}_h^{-1} [g] (x) = (2\pi)^{-1} \sum_{k \in \mathbb{T}_h^{*}}g(k)e^{ikx}.
\end{equation*}
Note the formal convergence as $h \rightarrow 0$ where $\mathbb{T}_h$ tends to $\mathbb{T} = [-\pi,\pi)$ and $\mathcal{F}_h$ tends to $\mathcal{F}$, the Fourier transform on $\mathbb{T}$.

The linear time evolution is governed by the integro-differential operator $(-\Delta_h)^{\frac{\alpha}{2}} := \mathcal{F}_h^{-1} \sigma_h (k) \mathcal{F}_h$ where $\sigma_h(\xi) = \left|\frac{2}{h}\sin\left(\frac{h \xi}{2}\right)\right|^{\alpha}$ for $\xi \in [-\frac{\pi}{h},\frac{\pi}{h})$. When $\alpha$ is an even integer, note that $(-\Delta_h)^{\frac{\alpha}{2}}$ is a local operator, exemplified in the simplest case of $\alpha = 2$ where $\Delta_h f (x) = \frac{f(x+h) + f(x-h) - 2f(x)}{h^2}$ is the center-difference discrete Laplacian. We do not treat $\alpha=4$ here; for the continuum biharmonic case see \cite{zhang2024blow} for local well-posedness and blow-up criteria proved using $X^{s,b}$ spaces, dyadic decompositions, the $TT^*$ method, and energy estimates, which are tools methodically relevant to this work.

Globally, the linear propagator is given by the unitary operator $U_h(t) := e^{-it(-\Delta_h)^{\frac{\alpha}{2}}}$ defined by the multiplier $e^{-it \sigma_h(k)}$. By convention when $h=0$, denote $U_0(t) = U(t) := e^{-it(-\Delta)^{\frac{\alpha}{2}}}$ where $(-\Delta)^{\frac{\alpha}{2}}$ is defined by the symbol $\sigma_0(k) = |k|^\alpha$ on the Fourier side. Recall that $U(t)$ is unitary on the Sobolev space $H^s(\mathbb{R})$ for any $s \in \mathbb{R}$. On the other hand, the discrete Sobolev space is defined with the norm
\begin{equation*}
    \| f \|_{H^s_h}^2 = \| \langle 
 \nabla_h\rangle^{s} f \|_{L^2(\mathbb{T})}^2:= \frac{1}{2\pi} \sum_{k \in \mathbb{T}_h^*} \langle k \rangle^{2s}|\mathcal{F}_h [f](k)|^2,
\end{equation*}
where $\langle \xi \rangle := (1+|\xi|^2)^{\frac{1}{2}}$, and similarly for $\| f \|_{H^s(\mathbb{T})}$.

To study dispersive smoothing, it is often useful to analyze the linear evolution of dyadic frequency components and sum each contribution, utilizing the orthogonality properties of the Littlewood-Paley operators. Throughout this paper, let $N \in 2^{\mathbb{Z}}$ satisfy $N_* \leq N \leq 1$ where $N_* = 2^{\lceil \log_2 (\frac{h}{\pi})\rceil - 1}$. Define
\begin{equation*}
    P_N =
    \begin{cases}
        \mathcal{F}_h^{-1}\chi_{\{|k| \in (\frac{\pi N}{2h},\frac{\pi N}{h}]\}} \mathcal{F}_h,\ & \text{if\ } 2 N_* \leq N \leq 1,\\
        Id - \sum\limits_{2N_* \leq N \leq 1} P_N, & \text{if\ } N = N_*,
    \end{cases}
\end{equation*}
where $Id$ is the identity operator and $\chi_E$ is the characteristic function on $E \subseteq [-\frac{\pi}{h},\frac{\pi}{h})$. As a shorthand, let $P_{\leq N} := \sum\limits_{M \leq N} P_M$. 

To relate continuum and discrete data, the operators $d_h$ (discretization) and $p_h$ (linear interpolation) are used in our approach where, given $f:\mathbb{T} \rightarrow \mathbb{C}$ and $g:\mathbb{T}_h \rightarrow \mathbb{C}$,
\begin{equation}\label{interpolation}
\begin{split}
d_h f (x) &= \frac{1}{h}\int_{x}^{x+h} f(x^\prime)dx^\prime,\ x \in \mathbb{T}_h,\\
p_h g(x) &= g(x_0) + \frac{g(x_0 + h) - g(x_0)}{h} (x-x_0),\ x_0 \in \mathbb{T}_h,\ x \in [x_0,x_0+h).
\end{split}
\end{equation}
See \cite{hong2019strong,hong2021finite} for more details on the properties of $d_h,p_h$ on Sobolev spaces.

We use the notation $f \lesssim g$ (or similarly $f \gtrsim g$) if $f \leq Cg$ for some universal constant $C>0$ and denote $f \simeq g$ by $f \lesssim g$ and $f \gtrsim g$. For more details on the harmonic analysis and numerical analysis on discrete spaces, see \cite{Younghun,hong2021finite}.

\section{Strichartz estimates.}\label{Strichartz_section}
The non-zero curvature of the dispersion relation yields dispersive smoothing estimates, or the Strichartz estimates, manifested as the boundedness of evolution time-dependent operators in various norms of Lebesgue spaces consisting of space-time functions. In \cite{hong2019strong}, the proof of the nonlocal continuum limit in the energy space $H^{\frac{\alpha}{2}}(\mathbb{R})$ when $\alpha \in (1,2)$ is based on the sharp Strichartz estimate
\begin{equation}\label{sharp_strichartz}
    \| e^{-it(-\Delta_h)^{\frac{\alpha}{2}}}f\|_{L^q_t(\mathbb{R};L^r_h)} \lesssim \| |\nabla_h|^{\frac{3-\alpha}{q}} f\|_{L^2_h},
\end{equation}
for $2 \leq q,r \leq \infty$ satisfying $\frac{3}{q} + \frac{1}{r} = \frac{1}{2}$. Since the approach taken to obtain \eqref{sharp_strichartz} does not translate directly to a compact domain, an alternative approach based on approximating an oscillatory sum with an oscillatory integral (\Cref{lemma:zygmund}) was used in \cite{hong2021finite,vega1992restriction}. Here we adopt their method and show (\Cref{strichartz})
\begin{equation}\label{nonsharp_strichartz}
    \| e^{-it(-\Delta_h)^{\frac{\alpha}{2}}}f \|_{L^q([0,1];L^r_h)} \lesssim_\epsilon \| f \|_{H^{\frac{2}{q}+\epsilon}_h}.
\end{equation} 

For $f \in L^1_h \setminus \{0\}$, $\| U_h(t)f \|_{L^\infty_h}$ cannot decay to zero as $t \rightarrow \infty$ due to the conservation of $L^2_h$ norm; see \Cref{lwp_fdnls}. However the discrete dispersive estimates hold locally in time.

\begin{proposition}\label{prop:main}
    Let $\alpha \in (1,2]$ and $|t| \leq \frac{\pi^{2-\alpha}}{2\alpha} \left(\frac{h}{N}\right)^{\alpha-1}$. Then
    \begin{equation}\label{dispersiveest}
        \| U_h(t) P_{\leq N} f\|_{L^\infty_h} \lesssim |\alpha - 1|^{-\frac{1}{3}} \left(\frac{N}{h}\right)^{1-\frac{\alpha}{3}}|t|^{-\frac{1}{3}} \| f \|_{L^1_h}.
    \end{equation}
\end{proposition}

Our proof of \Cref{prop:main} is motivated from \cite{hong2021finite,vega1992restriction} where a discrete oscillatory sum (see \eqref{bottleneck}) is approximated by an oscillatory integral, after which the Van der Corput Lemma is applied. 

\begin{lemma}[{\cite[Chapter 5, Lemma 4.4]{zygmund2002trigonometric}}]\label{lemma:zygmund}
Let $a<b$ and $0 < \epsilon < 1$. Assume $\sup\limits_{\xi \in (a,b)}|\phi^\prime(\xi)| \leq 2\pi(1-\epsilon)$ and $\phi^\prime$, monotonic in $(a,b)$. Then there exists $A_\epsilon>0$ independent of $a,b,\phi$ such that
\begin{equation*}
    \left|\int_a^b e^{i \phi(\xi)} d\xi - \sum_{a < k \leq b} e^{i\phi(k)}\right| \leq A_\epsilon.
\end{equation*}
\end{lemma}

\begin{proof}[Proof of \Cref{prop:main}]
Consider the identity
\begin{equation*}
\begin{split}
    U_h(t) P_{\leq N} f (x) &= \frac{1}{2\pi} \sum_{|k| \leq \frac{\pi N}{h}} e^{i\left(-t \left| \frac{2}{h}\sin \frac{hk}{2}\right|^{\alpha} + kx\right)} \mathcal{F}_h f(k)\\
    &=h \sum_{x^\prime \in \mathbb{T}_h} f(x^\prime) \sum_{|k| \leq \frac{\pi N}{h}}\frac{1}{2\pi} e^{i\left(-t \left| \frac{2}{h}\sin \frac{hk}{2}\right|^{\alpha} + k(x-x^\prime)\right)} = K_{t} \ast f,
\end{split}
\end{equation*}
where $K_t(x):= \sum\limits_{|k| \leq \frac{\pi N}{h}}\frac{1}{2\pi} e^{i\left(-t \left| \frac{2}{h}\sin \frac{hk}{2}\right|^{\alpha} + kx\right)}$ and $
\ast$ denotes the discrete convolution defined by the measure $d\mu_h$. As a shorthand, let 
\begin{equation}\label{phase}
\phi(\xi) = -t \left| \frac{2}{h}\sin \frac{h\xi}{2}\right|^{\alpha} + \xi x.    
\end{equation}
It suffices to show 
\begin{equation}\label{bottleneck}
\| K_t \|_{L^\infty_h} \lesssim |\alpha - 1|^{-\frac{1}{3}} \left(\frac{N}{h}\right)^{1-\frac{\alpha}{3}}|t|^{-\frac{1}{3}}    
\end{equation}
by the Young's inequality. By the triangle inequality,
\begin{equation}\label{bottleneck2}
    |K_t(x)| \leq \left|K_t(x) - \frac{1}{2\pi}\int_{-\frac{\pi N}{h}}^{\frac{\pi N}{h}} e^{i\phi(\xi)}d\xi\right| + \left| \frac{1}{2\pi}\int_{-\frac{\pi N}{h}}^{\frac{\pi N}{h}} e^{i\phi(\xi)}d\xi \right|=: I + II.
\end{equation}
To show that $I = O(1)$ and $II$ is consistent with \eqref{bottleneck} by \Cref{lemma:zygmund} and the Van der Corput Lemma, respectively, the higher order derivatives of $\phi$ need to estimated. Let $sgn(\xi) = 1$ if $\xi > 0$ and $sgn(\xi) = -1$ if $\xi < 0$. Then,
\begin{equation}
\begin{split}
\phi^\prime(\xi) &= -\alpha t \cdot sgn(\xi)\cos\left(\frac{h\xi}{2}\right) \left| \frac{\sin(\frac{h\xi}{2})}{h/2} \right|^{\alpha-1}  + x,\\
\phi^{\prime\prime}(\xi) &= \alpha t \left(\frac{h}{2}\right)^{2-\alpha} \frac{1-\alpha \cos^2\left(\frac{h\xi}{2}\right)}{\left| \sin \left(\frac{h\xi}{2}\right)\right|^{2-\alpha}},\\
\phi^{\prime\prime\prime}(\xi) &= - \alpha t \left(\frac{h}{2}\right)^{3-\alpha} sgn(\xi) \frac{\left(\alpha^2 \cos^2 \left(\frac{h \xi}{2}\right) - 3\alpha +2\right)\cos\left(\frac{h\xi}{2}\right)}{\left| \sin \left(\frac{h\xi}{2}\right)\right|^{3-\alpha}}.
\end{split}
\end{equation}

By direct computation, $\phi^\prime$ is monotonic on $E_1 := (-\xi_0,\xi_0)$ and $E_2 := [-\frac{\pi}{h},\frac{\pi}{h}]\setminus (-\xi_0,\xi_0)$ separately where $\xi_0 := \frac{2}{h} \arccos (\alpha^{-\frac{1}{2}})$ is the unique positive root of $\phi^{\prime\prime}$ in $(0,\frac{\pi}{h})$. By choosing $\epsilon = \frac{1}{4}$ in \Cref{lemma:zygmund}, it can further be verified that $\sup\limits_{\xi \in [-\frac{\pi N}{h},\frac{\pi N}{h}]}|\phi^\prime(\xi)| \leq \frac{3\pi}{2}$ given the restriction on $|t|$, which shows $I = O(1)$ estimating the difference on $E_1,E_2$ separately.

To estimate the integral in $II$, the lower bounds of $|\phi^{\prime\prime}|,|\phi^{\prime\prime\prime}|$ are estimated. Let
\begin{equation*}
S = \Bigl\{|\xi| \leq \frac{\pi N}{h}: \left|1-\alpha \cos^2 \left(\frac{h\xi}{2}\right)\right| \geq \frac{|\alpha - 1|}{2}\Bigr\}.   
\end{equation*}
On $S$, the bound $\left|\sin \left(\frac{h\xi}{2}\right)\right| \leq \frac{h|\xi|}{2}$ is used to obtain
\begin{equation}\label{vdc1}
|\phi^{\prime\prime}(\xi)| \gtrsim (\alpha - 1)|t| |\xi|^{-(2-\alpha)}.   
\end{equation}
On $[-\frac{\pi N}{h},\frac{\pi N}{h}]\setminus S$, we have
\begin{equation*}
    |\phi^{\prime\prime\prime}(\xi)| \gtrsim (\alpha - 1) |t| |\xi|^{-(3-\alpha)},
\end{equation*}
since
\begin{equation*}
    \left|\alpha^2 \cos^2 \left(\frac{h\xi}{2}\right)-3 \alpha +2\right| \geq 2 |\alpha - 1| - \left|\alpha^2 \cos^2\left(\frac{h\xi}{2}\right)-\alpha\right| \geq 2|\alpha - 1| - \frac{\alpha}{2} |\alpha - 1| \geq \alpha - 1,
\end{equation*}
and
\begin{equation*}
    \left|\alpha \cos^2\left(\frac{h\xi}{2}\right)\right| \geq 1 - \left|\alpha \cos^2\left(\frac{h\xi}{2} \right)- 1\right| \geq 1 - \frac{|\alpha - 1|}{2} \geq \frac{1}{2}.
\end{equation*}
Hence
\begin{equation}\label{vdc}
\begin{split}
\left|\int_{-\frac{\pi N}{h}}^{\frac{\pi N}{h}} e^{i\phi(\xi)}d\xi\right| &\leq \left|\int_S e^{i\phi(\xi)}d\xi\right| + \left|\int_{[-\frac{\pi N}{h},\frac{\pi N}{h}]\setminus S} e^{i\phi(\xi)}d\xi\right|\\
&\lesssim \max\left((\alpha - 1)^{-\frac{1}{2}}\left(\frac{N}{h}\right)^{1-\frac{\alpha}{2}}|t|^{-\frac{1}{2}},(\alpha - 1)^{-\frac{1}{3}}\left(\frac{N}{h}\right)^{1-\frac{\alpha}{3}}|t|^{-\frac{1}{3}}\right)\\
&\lesssim (\alpha - 1)^{-\frac{1}{3}}\left(\frac{N}{h}\right)^{1-\frac{\alpha}{3}}|t|^{-\frac{1}{3}},
\end{split}
\end{equation}
where the last inequality follows from interpolating $\left|\int_{-\frac{\pi N}{h}}^{\frac{\pi N}{h}} e^{i\phi(\xi)}d\xi\right| \lesssim \frac{N}{h}$ and the Van der Corput estimate obtained from \eqref{vdc1}. Substituting $I = O(1)$ and \eqref{vdc} into \eqref{bottleneck2}, the desired estimate \eqref{bottleneck} is shown since for $|t| \lesssim (\frac{h}{N})^{\alpha - 1}$ and $N \geq N_*$,
\begin{equation*}
(\alpha - 1)^{-\frac{1}{3}}\left(\frac{N}{h}\right)^{1-\frac{\alpha}{3}}|t|^{-\frac{1}{3}} \gtrsim (\alpha - 1)^{-\frac{1}{3}}\left(\frac{N}{h}\right)^{\frac{2}{3}} \gtrsim (\alpha - 1)^{-\frac{1}{3}}.     
\end{equation*}   
\end{proof}

The forthcoming $TT^*$ method yielding short-time Strichartz estimates is standard in the dyadic framework; compare the Kerr-law NLS trilinear/$TT^*$ machinery \cite{zhang2021trilinear} and the related low-regularity analyses for generalized Kawahara/Ostrovsky \cite{zhang2018well,zhang2016well}, as well as higher-order negative-index variants \cite{zhang2019low}. Our periodic lattice phase, with weakened curvature and aliasing at certain dyadic scales, necessitates the time localization and oscillatory–sum bounds established above.

\begin{corollary}\label{strichartz}
    Let $(q,r)\in [2,\infty]^2$ be lattice-admissible and $\alpha \in (1,2]$. Then for all $\epsilon>0$,    \begin{align}
    \| U_h(t) f \|_{L^q([0,1];L^r_h)} &\lesssim_\epsilon |\alpha - 1|^{-\frac{1}{6}(1-\frac{2}{r})} \| f \|_{H^{\frac{2}{q}+\epsilon}_h},\label{strichartzest}\\
    \left|\left| \int_0^t U_h(t-t^\prime) F(t^\prime)dt^\prime \right|\right|_{L^q([0,1];L^r_h)} &\lesssim_\epsilon |\alpha - 1|^{-\frac{1}{3}(1-\frac{2}{r})}\| F \|_{L^1([0,1];H^{\frac{2}{q}+\epsilon}_h)}.\label{strichartzest2}    
    \end{align}
\end{corollary}

\begin{proof}
    Observe that $\tilde{U}_h(t) = P_{\leq N} U_h(t(\frac{N}{h})^{3-\alpha})P_{\leq N}$ satisfies the hypothesis of \cite[Theorem 1.2]{keel1998endpoint}, and therefore
\begin{equation}\label{strichartz2}
    \| U_h(t) P_{\leq N} f \|_{L^q([0,\tau];L^r_h)} \lesssim |\alpha - 1|^{-\frac{1}{6}(1-\frac{2}{r})} \left(\frac{N}{h}\right)^{\frac{3-\alpha}{q}} \| f \|_{L^2_h}, 
\end{equation}
where $\tau = \frac{\pi^{2-\alpha}}{2\alpha} \left(\frac{h}{N}\right)^{\alpha-1}$. By iterating \eqref{strichartz2} on $[0,1]$ using the unitarity of $U_h(t)$, \eqref{strichartzest} is shown. Another application of \cite[Theorem 1.2]{keel1998endpoint} yields the inhomogeneous estimate \eqref{strichartzest2}. Note that the implicit constant of \eqref{strichartzest2} is that of \eqref{strichartzest} squared by the TT$^*$ argument based on the complex interpolation of dispersive estimates between $\| U_h(t) P_{\leq N} f \|_{L^2_h} \leq \| f \|_{L^2_h}$ and \eqref{dispersiveest} given by
\begin{equation*}
    \| U_h(t) P_{\leq N} f \|_{L^r_h} \lesssim |\alpha - 1|^{-\frac{1}{3}(1-\frac{2}{r})} \left(\frac{N}{h}\right)^{(1-\frac{\alpha}{3})(1-\frac{2}{r})} |t|^{-\frac{1}{3}(1-\frac{2}{r})} \| f \|_{L^{\frac{r}{r-1}}_h}.
\end{equation*}
\end{proof}

\begin{remark}
Our proof of \Cref{prop:main} and \Cref{strichartz} is adapted from \cite{hong2021finite}. The proof of \Cref{prop:main} differs from that of \cite{hong2021finite} in that we use $\epsilon > 0$ when applying \Cref{lemma:zygmund} whereas \cite{hong2021finite} used $\epsilon = 0$, which may lead to the error bound $A_\epsilon$ to blow up. This subtlety could easily be circumvented by choosing $0 < \epsilon < \frac{1}{2}$. Furthermore note that the right-hand side (RHS) of \eqref{strichartzest} is measured in $H^{\frac{2}{q}+\epsilon}_h$ whose Sobolev regularity is independent of $\alpha$. The $\epsilon>0$ is a result of the non-endpoint Sobolev embedding.
\end{remark}

Without the derivative loss of $1-\frac{\alpha}{3}$ in \Cref{prop:main}, it can be shown that the family $\{U_h(t)\}_{h>0}$ fails to be uniformly bounded. Our approach modifies the argument of Ignat and Zuazua in \cite{ignat2009numerical} based on the estimation of the oscillatory integral. The authors utilized a result in \cite{magyar2002discrete}, which proved the norm equivalence of band-limited functions with continuous Fourier transform. Since the functions on periodic lattices cannot be identified with the aforementioned functions studied in \cite{magyar2002discrete}, we cannot directly use the norm-equivalence result. Nonetheless, the discrete dispersive estimates are not uniformly bounded in the periodic setting, and we give a sketch of proof for completion.

\begin{proposition}\label{blow-up}
    For any $T>0,\ \alpha \in (1,2]$, and $1 \leq p < q \leq \infty$, let
    \begin{equation}\label{h_hypothesis}
    0 < h \ll \min(T^{-\frac{1}{3-\alpha}}, T^{\frac{1}{\alpha}} |\alpha - 1|^{\frac{3(2-\alpha)}{2\alpha}},|\alpha - 1|^{\frac{2-\alpha}{2}}).    
    \end{equation}
Then there exists $f \in L^1_h$, depending on the previous parameters, such that\begin{equation*}
        \frac{\| U_h(T) f\|_{L^q_h}}{\| f \|_{L^p_h}} \geq T^{-\frac{1}{3}(\frac{1}{p} - \frac{1}{q})} h^{-(1-\frac{\alpha}{3})(\frac{1}{p} - \frac{1}{q})}.
    \end{equation*}
\end{proposition}
\begin{proof}

Given $f \in L^1_h$, consider the pullback $\tilde{f} \in L^1(\mathbb{Z}_{2M})$ where $\tilde{f}(x) = f(hx)$ for $x \in \mathbb{Z}_{2M} = \mathbb{Z}/(2M \mathbb{Z})$. Then the pullback of $U_h(T)f$ is $U_1(\frac{T}{h^\alpha}) \tilde{f}$ and
\begin{equation}\label{scaling}
    \frac{\| U_h(T) f\|_{L^q_h}}{\| f \|_{L^p_h}} = h^{-(\frac{1}{p} - \frac{1}{q})}\frac{\| U_1(\frac{T}{h^\alpha}) \tilde{f}\|_{l^q(I_M)}}{\| \tilde{f} \|_{l^p(I_M)}},
\end{equation}
where we identify the periodic functions on $\mathbb{Z}_{2M}$ with their restrictions on the fundamental domain $I_M = [-M,M) \cap \mathbb{Z}$ and drop the tilde notation henceforth.

Let $t = \frac{T}{h^\alpha}$ and $\tau > 0$ satisfy $\tau \gg |\alpha-1|^{-\frac{3}{2}(2-\alpha)}$ and $\tau \simeq t$. Define $\widehat{\psi_\tau} (\xi) = \tau^{\frac{1}{3}} \widehat{\psi}(\tau^{\frac{1}{3}}(\xi - \xi_c))$ where $\widehat{\psi} \in C^\infty_c (-1,1)$ is non-negative and $\int \widehat{\psi} = 1$; note that $\xi_c = 2 \cos^{-1} (\alpha^{-\frac{1}{2}}) > 0$ is the inflection point of $\phi$ with $h=1$, defined in \eqref{phase}. For $f_{\tau,M} := \psi_\tau \chi_{I_M}$, we have $\| f_{\tau,M}\|_{l^p(I_M)} \leq \| \psi_\tau \|_{l^p} \simeq \tau^{\frac{1}{3p}}$ where the implicit constant depends only on $\psi$. 

We claim that $|(U_1(t)f_{\tau,M})(x)| \gtrsim 1$ whenever $x \in \mathbb{Z}$ satisfies $|x - t \sigma_1^\prime(\xi_c)| \leq \frac{\tau^{\frac{1}{3}}}{8}$. The conclusion of \Cref{blow-up} follows immediately since $\| U_1(t)f_{\tau,M}\|_{l^q(I_M)} \gtrsim \tau^{\frac{1}{3q}}$. The linear evolution is given by
\begin{equation*}
    2\pi (U_1(t) f_{\tau,M})(x) = \int_{-\pi}^{\pi} e^{i\phi(\xi)}\widehat{\psi_\tau} (\xi)d\xi - \int_{-\pi}^{\pi} e^{i\phi(\xi)}\left(\widehat{\psi_\tau} (\xi) - \mathcal{F}[f_{\tau,M}](\xi)\right)d\xi.
\end{equation*}
By the Mean Value Theorem, the first term is bounded below as
\begin{align}\label{osc_int}
\left| \int_{- \pi}^{\pi} e^{i\phi(\xi)}\widehat{\psi_\tau} (\xi)d\xi \right| &\geq 1 - 2\tau^{-\frac{1}{3}} \sup_{\xi^\prime \in supp(\widehat{\psi_\tau})} |x - t \sigma_1^\prime(\xi^\prime)|\nonumber\\
&\geq 1 - 2\tau^{-\frac{1}{3}} |x - t \sigma_1^\prime(\xi_c)| - \left|\frac{t}{\tau}\right| \sup_{\zeta \in supp(\widehat{\psi_\tau})} |\sigma_1^{(3)}(\zeta)|.
\end{align}
By direct computation, it can be shown that $|\sigma_1^{(4)}(\cdot)|$ is monotonically decreasing on $(0,\xi_c)$, and therefore
\begin{equation*}
    |\sigma_1^{(3)}(\zeta)| \leq |\sigma_1^{(3)}(\xi_c)| + 2\tau^{-\frac{1}{3}} |\sigma_1^{(4)}(\xi_c - \tau^{-\frac{1}{3}})| = O(1),
\end{equation*}
where the upper bound, independent of $\alpha$, follows from the bound on $\tau$ and \eqref{h_hypothesis}; this calculation consists of standard algebra, and thus is omitted for the sake of exposition. Since $\tau \gg 1$, it follows that the RHS of \eqref{osc_int} $\geq \frac{1}{2}$. On the other hand, using $\sup\limits_{x \in \mathbb{R}}\langle x \rangle^{4}|\psi(x)| = O_{\psi}(1)$, the second term is bounded above as
\begin{equation*}
    \left| \int_{-\pi}^{\pi} e^{i\phi(\xi)}\left(\widehat{\psi_\tau} (\xi) - \mathcal{F}[f_{\tau,M}](\xi)\right)d\xi \right| \leq \sum_{x \in \mathbb{Z} \setminus I_M} |\psi(\tau^{-\frac{1}{3}}x)| \lesssim \tau^{-\frac{3}{\alpha} + \frac{4}{3}} \leq \tau^{-\frac{1}{6}},
\end{equation*}
by \eqref{h_hypothesis} where the implicit constant can be chosen uniformly in $\alpha$, and this shows the claim.

\end{proof}

\section{Convergence as $h \rightarrow 0$.}\label{proof_thm1}
The proof of \Cref{result} is presented. Our strategy is to directly estimate the difference between the solutions in $\mathbb{T}$ and $\mathbb{T}_h$. To address the subtlety that $u(t)$ and $u_h(t)$ are defined on different spaces, we lift $u_h(t)$ via linear interpolation. Since there is no canonical way to interpolate discrete data into continuum data or vice versa via discretization, there is flexibility in how the error is defined and computed. For example, the exact solution whose spatiotemporal frequency is concentrated at a single site admits linear convergence \eqref{exact_cts} in $L^2(\mathbb{T})$ and quadratic convergence \eqref{exact_discrete} in $L^2_h$. Whether or not other methods of error estimation yield similar results as \Cref{result} is left for further research. 

The technical lemmas used to prove the theorem are adapted from \cite{hong2021finite,hong2019strong} either directly or with minimal modifications. Though our main result focuses on $\alpha \in (1,2]$, we give a general statement.

\begin{lemma}\label{l0}
       
    Let $\alpha > 0,\ 0 \leq s \leq 2 + \alpha,\ 0 \leq \tau \leq t$, and $p > 1$. Then
    \begin{align}
        \| p_h U_h(t) u_{h,0} - U(t) u_0\|_{L^2(\mathbb{T})} &\lesssim |t| h^{\frac{2s}{2+\alpha}} (\| u_{h,0} \|_{H^s_h} + \| u_0 \|_{H^s(\mathbb{T})}) + \| p_h u_{h,0} - u_0\|_{L^2(\mathbb{T})},\label{l1}\\
        \| (p_h U_h(t-\tau) - U(t-\tau)p_h)F(\tau)\|_{L^2(\mathbb{T})} &\lesssim |t-\tau| h^{\frac{2s}{2+\alpha}} \| F(\tau) \|_{H^s_h(\mathbb{T})},\label{l2}\\
        \| p_h(|u_h|^{p-1}u_h) - |p_h u_h|^{p-1}p_h u_h\|_{L^2(\mathbb{T})} &\lesssim h^s \| u_h \|_{L^\infty_h}^{p-1} \| u_h \|_{H^s_h},\label{l3}\\
        \| p_h d_h f - f\|_{L^2(\mathbb{T})} &\lesssim h^{s} \| f \|_{H^{s}(\mathbb{T})}.\label{l4}
    \end{align}
\end{lemma}

\begin{proof}[Proof of \Cref{result}]
In this proof, $\| \cdot \|$ denotes the norm in $L^2(\mathbb{T})$. Given $u_0 \in H^s(\mathbb{T})$, let $u(t) = S(t)u_0$ and $u_h(t) = S_h(t)d_h u_0$. It can be directly verified that $d_h$ defines a bounded linear operator from $L^2(\mathbb{T})$ to $L^2_h$ and $H^1(\mathbb{T})$ to $H^1_h$, and consequently it follows from interpolating the two estimates that $d_h:H^{s^\prime}(\mathbb{T}) \rightarrow H^{s^\prime}_h$ is bounded for any $s^\prime \in [0,1]$. Therefore the time of existence in \Cref{lwp_fdnls} is bounded from below since $T_h \sim_\alpha \| d_h u_0 \|_{H^s_h}^{-3} \gtrsim \| u_0 \|_{H^s}^{-3}$. Let $T^\prime(\| u_0 \|_{H^s})>0$ be the time of existence stated in \Cref{lwp_fnls} and take
\begin{equation}\label{local_time}
T:=\min\{T^\prime,\inf\limits_{h} T_h\} > 0.    
\end{equation}
By the triangle inequality and \eqref{l1}, \eqref{l2}, and \eqref{l3}, we have
\begin{align}
\| (p_h S_h(t) d_h  - S(t)) u_0\| &\leq \| (p_h U_h(t) d_h - U(t)) u_0\| +\int_0^t \| (p_h U_h(t-\tau) - U(t-\tau)p_h)(|u_h|^2 u_h)(\tau)\| d\tau \nonumber\\
&+ \int_0^t \| p_h(|u_h|^{2}u_h) - |p_h u_h|^{2}p_h u_h\| d\tau + \int_0^t \| |p_h u_h|^2 p_h u_h - |u|^2 u\| d\tau \nonumber\\
&=: I_1 + I_2 + I_3 + I_4 \label{decomposition}\\ 
\lesssim (1+|T|) h^{\frac{2s}{2+\alpha}}\| u_0 \|_{H^s} &+\int_0^t (|t-\tau| h^{\frac{2s}{2+\alpha}} + h^s) \| u_h \|_{L^\infty_h}^2 \| u_h \|_{H^s_h} d\tau + \int_0^t (\| p_h u_h \|_{L^\infty}^2 + \| u \|_{L^\infty}^2) \| p_h u_h - u \| d\tau,\label{estimate1}
\end{align}
having used the discrete Sobolev estimate $\| |u_h|^2 u_h\|_{H^s_h} \lesssim \| u_h \|_{L^\infty}^2 \| u_h \|_{H^s_h}$. Similar to the proof of \Cref{lwp_fdnls} where $X_T = C([0,T_h];L^2_h) \cap L^6([0,T_h];L^\infty_h)$, the nonlinear terms are estimated as 
\begin{equation*}
    \left|\left| \| u_h \|_{L^\infty_h}^2 \| u_h \|_{H^s_h} \right|\right|_{L^1_T} \leq T_h^{\frac{2}{3}} \| u_h \|_{X_T}^3 \lesssim \| d_h u_0 \|_{H^s_h} \lesssim \| u_0 \|_{H^s}.
\end{equation*}
Furthermore a direct estimation on the linear interpolation yields $\| p_h u_h \|_{L^\infty} \lesssim \| u_h \|_{L^\infty_h}$ since for $x_0 \in \mathbb{T}_h$ and $x \in [x_0,x_0 + h)$,
\begin{equation*}
    |p_h u_h (x)| \leq |u_h(x_0)| + \frac{|u_h(x_0+h)|+|u_h(x_0)|}{h} |x-x_0| \leq 3 \| u_h \|_{L^\infty_h}.
\end{equation*}
Altogether, 
\begin{equation*}
\eqref{estimate1} \lesssim (1+|T|) h^{\frac{2s}{2+\alpha}} \| u_0 \|_{H^s} + \int_0^t (\| u_h \|_{L^\infty_h}^2 + \| u \|_{L^\infty}^2) \| p_h u_h - u \| d\tau,    
\end{equation*}
and the Gronwall's inequality for $t \in [0,T]$ yields
\begin{equation*}
    \| p_h u_h(t) - u(t) \| \lesssim (1+|T|) h^{\frac{2s}{2+\alpha}} \| u_0 \|_{H^s} e^{\int_0^T \| u_h(\tau) \|_{L^\infty_h}^2 + \| u(\tau) \|_{L^\infty}^2 d\tau}. 
\end{equation*}
It suffices to show that there exists $C(\| u_0 \|_{H^s},\alpha)>0$ independent of $h$ such that $\| u_h \|_{L^2_T L^\infty_h}^2 + \| u \|_{L^2_T L^\infty}^2 \leq C$. By \Cref{lwp_fdnls},
\begin{equation*}
    \| u_h \|_{L^2_T L^\infty_h}^2 \leq T_h^{\frac{2}{3}} \| u_h \|_{X_T}^2 = O(1)
\end{equation*}
is independent of $h$. To show $\| u \|_{L^2_T L^\infty}^2 = O(1)$, note that
\begin{equation}\label{estimate3}
    \| u \|_{L^2_T L^\infty}^2 \leq T^{\frac{1}{2}} \| u \|_{L^4_T L^\infty}^2 \lesssim T^{\frac{1}{2}} \| u \|_{L^4_T W^{\frac{1}{4}+\epsilon^\prime,4}}^2,
\end{equation}
by the H\"older's inequality and the Sobolev embedding $W^{\frac{1}{4}+\epsilon^\prime,4}(\mathbb{T}) \hookrightarrow L^\infty(\mathbb{T})$ where $0 < \epsilon^\prime < s - \frac{3-\alpha}{4}$. Then \Cref{bilinear} is applied at the minimum regularity $\frac{2-\alpha}{4}$ and $0 < \epsilon \ll 1$ to obtain
\begin{equation*}
    \eqref{estimate3} \lesssim T^{\frac{1}{2}} \| u \|_{X_T^{\frac{1}{4}+\epsilon^\prime,\frac{1}{2}-\epsilon}} \| u \|_{X_T^{\frac{3-\alpha}{4}+\epsilon^\prime,\frac{1}{2}+\epsilon}} \lesssim T^{\frac{1}{2}}\| u \|_{X_T^{\frac{3-\alpha}{4}+\epsilon^\prime,\frac{1}{2}+\epsilon}}^2 \lesssim T^{\frac{1}{2}} \| u \|_{C_T H^s}^2 \lesssim T^{\frac{1}{2}} \| u_0 \|_{H^s}^2,
\end{equation*}
where the second and the third estimates reflect the embeddings $X_T^{s_1,b_1} \hookrightarrow X_T^{s_2,b_2}$, for $s_1 \geq s_2,\ b_1 \geq b_2$, and $X_T^{s,\frac{1}{2}+\epsilon} \hookrightarrow C([0,T];H^{s}(\mathbb{T}))$, respectively, and the last inequality follows from the proof of \Cref{lwp_fnls} based on the fixed point argument.
\end{proof}

\begin{remark}
Since there is no canonical way to define a numerical error, the convergence rate depends on the method of discretization and interpolation. As \eqref{interpolation}, we discretize data on a smooth domain by averaging over an interval of length $h$ and linearly interpolating discrete data. Though we do not take the following approach, \cite{bernier2019bounds,chauleur:hal-04142120} considered the pointwise projection of continuous functions onto $h\mathbb{Z}$ and the Shannon interpolation that takes the discrete convolution of discrete data against the sinc function. It is commented in \cite{chauleur:hal-04142120} that the Shannon interpolation is better suited to show convergence in higher Sobolev norms. It is of interest to show the continuum limit of our model in higher Sobolev norms, which would require a uniform-in-$h$ control of the Sobolev norms of discrete solutions. However the method of modified energy used in the previous references to obtain bounds on higher Sobolev norms is not directly applicable in our nonlocal case due to the complexity of the Leibniz rule for fractional derivatives.
\end{remark}

In \Cref{sharpness}, we combine the short-time dispersive control with an energy estimate tailored to the discrete-to-continuum interface to obtain the sharp rate $h^{\frac{\alpha}{2+\alpha}}$ for the error in the energy space $H^{\frac{\alpha}{2}}(\mathbb{T})$. The role of energy here parallels its use in fractional dispersive models to establish sharp global well-posedness results \cite{wang2018sharp}; in contrast, our objective is a rate-optimal local approximation calibrated to the periodic fractional lattice, rather than threshold optimality. A family of plane wave solutions in \Cref{appendix_exact} satisfies linear convergence (see \eqref{exact_cts}), and hence do not satisfy an estimate of the form \eqref{sharpness_estimate}. We first note that the mass and energy conservation for sufficiently regular data admits the global result.

\begin{proposition}\label{result2}
    Let $\alpha \in (1,2]$ and $u_0 \in H^{\frac{\alpha}{2}}(\mathbb{T})$. Then there exists $C_1,C_2 > 0$ depending only on $\| u_0 \|_{H^{\frac{\alpha}{2}}}$ and $\alpha$ such that the error estimate
    \begin{equation}\label{continuum_global}
        \| p_h S_h(t) d_h u_0 - S(t) u_0\|_{L^2(\mathbb{T})} \leq C_1  h^{\frac{\alpha}{2+\alpha}}e^{C_2 |t|} (1 + \| u_0 \|_{H^{\frac{\alpha}{2}}})^3,
    \end{equation}
    holds for all $t \in \mathbb{R}$.
\end{proposition}
\begin{proof}
With initial data of finite energy, the Sobolev embedding $H^{\frac{\alpha}{2}}(\mathbb{T}) \hookrightarrow L^\infty(\mathbb{T})$ allows a more straightforward proof, without resorting to the Strichartz estimates, than the one presented in \Cref{proof_thm1}.    
\end{proof}

It is of interest to show the existence of $u_0 \in H^{\frac{\alpha}{2}}(\mathbb{T})$ such that
\begin{equation*}
    C_1(t,\| u_0 \|_{H^{\frac{\alpha}{2}}}) h^{\frac{2}{2+\alpha}} \leq \| p_h S_h(t) d_h u_0 - S(t) u_0\|_{L^2} \leq C_2(t,\| u_0 \|_{H^{\frac{\alpha}{2}}}) h^{\frac{2}{2+\alpha}},
\end{equation*}
for any $t \in \mathbb{R}$. Instead we derive a partial result that is local in time, which shows the sharpness of the convergence rate $\frac{\alpha}{2+\alpha}$ in \eqref{continuum_global}. 
\begin{proposition}\label{sharpness}
    Let $\alpha \in (1,2]$. There exists $0< T \ll 1$ such that for any $0< h \ll 1$, we have $u_0^{h} \in H^{\frac{\alpha}{2}}(\mathbb{T})$ with $\| u_0^{h} \|_{H^{\frac{\alpha}{2}}} = O(1)$, independent of $h$, satisfying    \begin{equation}\label{sharpness_estimate}
        \| p_h S_h(t) d_h u_0^h - S(t) u_0^h\|_{L^\infty_T L^2} \geq c(T,\alpha) h^{\frac{2}{2+\alpha}},
    \end{equation}
    where $c>0$ is independent of $h>0$.
\end{proposition}

\begin{proof}
We fix $\alpha$ once and for all and all constants resulting from the Sobolev embedding or the Strichartz estimates are considered as implicit constants. Fix $T>0$ as \eqref{local_time} corresponding to $\{u_0: \| u_0 \|_{H^{\frac{\alpha}{2}}} = O(1)\}$. Recall from the proof of \Cref{lwp_fdnls} that $\| u_h \|_{L^\infty_T H^{\frac{\alpha}{2}}} \lesssim \| d_h u_0 \|_{H_{h}^{\frac{\alpha}{2}}}$ where an explicit description of $u_0$ is given in \eqref{explicit_construction}. Let $t \in [0,T]$ and $h \ll_\alpha 1$. From \eqref{decomposition},
\begin{equation*}
\begin{split}
I_2(t) + I_3(t) &\lesssim \int_0^t (|t-\tau| h^{\frac{\alpha}{2+\alpha}} + h^\frac{\alpha}{2}) \| u_h \|_{L^\infty_h}^2 \| u_h \|_{H_h^{\frac{\alpha}{2}}} d\tau \lesssim \| u_h \|^3_{L^\infty_T H_h^{\frac{\alpha}{2}}}\int_0^t (|t-\tau| h^{\frac{\alpha}{2+\alpha}} + h^\frac{\alpha}{2}) d\tau\\
&\lesssim \| d_h u_0 \|_{H^{\frac{\alpha}{2}}_h}^3 (t^2 h^{\frac{\alpha}{2+\alpha}} + |t| h^{\frac{\alpha}{2}}) \lesssim \| u_0 \|_{H^{\frac{\alpha}{2}}}^3 T h^{\frac{\alpha}{2+\alpha}},  
\end{split}
\end{equation*}
where we shrink $T$, if necessary, to admit the last inequality. On the other hand,
\begin{equation*}
\begin{split}
\int_0^t (\| p_h u_h \|_{L^\infty}^2 + \| u \|_{L^\infty}^2) \| p_h u_h - u \| d\tau &\leq (\| p_h u_h \|_{L^2([0,t];L^\infty)}^2 + \| u \|_{L^2([0,t];L^\infty)}^2) \| p_h u_h - u \|_{L^\infty([0,t];L^2)}\\
&\lesssim \| p_h u_h - u \|_{L^\infty_T L^2},
\end{split}
\end{equation*}
since $\| p_h u_h \|_{L^2([0,t];L^\infty)}^2 + \| u \|_{L^2([0,t];L^\infty)}^2 = O(1)$ as in the proof of \Cref{result}. Hence by the triangle inequality, we have
\begin{equation}\label{sharpness_lb}
\| p_h S_h(\cdot) d_h u_0 - S(\cdot) u_0\|_{L^\infty_T L^2} \gtrsim \| p_h U_h(t) d_h u_0 - U(t) u_0\|_{L^2} - \| u_0 \|_{H^{\frac{\alpha}{2}}}^3 T h^{\frac{\alpha}{2+\alpha}}.
\end{equation}
Observe that
\begin{equation}\label{reverse_triangle}
    \| p_h U_h(t) d_h u_0 - U(t) u_0\|_{L^2} \geq \| p_h U_h(t) d_h u_0 - U(t) p_h d_h u_0\|_{L^2} - \| U(t)(p_h d_h u_0 - u_0)\|_{L^2} =: A_1 - A_2,
\end{equation}
where $A_2 \lesssim h^{\frac{\alpha}{2}} \| u_0 \|_{H^{\frac{\alpha}{2}}}$ by the unitarity of $U(t)$ and \eqref{l4}. By the Plancherel's Theorem,
\begin{equation}\label{A0_estimate}
    \sqrt{2\pi} A_1 = \| (e^{-it\left|\frac{2}{h}\sin \frac{hk}{2}\right|^{\alpha}} - e^{-it|k|^\alpha}) \mathcal{F}[p_h d_h u_0]\|_{l^2}.
\end{equation}
Recalling from \cite[Lemma 5.5]{hong2019strong} that $p_h$ is a Fourier multiplier with the symbol $P_h(k) = \left(\frac{\sin(\frac{hk}{2})}{\frac{hk}{2}}\right)^2$ on $\mathbb{Z}$, or equivalently $\mathcal{F}[p_h f](k) = P_h(k) \mathcal{F}_h [f] (k)$, we have $\mathcal{F}[p_h U_h(t) d_h u_0](k) = e^{-it\left|\frac{2}{h}\sin \frac{hk}{2}\right|^{\alpha}}\mathcal{F}[p_h d_h u_0](k)$ for any $k \in \mathbb{Z}$; although a slight abuse of notation is used here, i.e., $k \in \mathbb{T}_h^*$ or $k \in \mathbb{Z}$ depending on the context, since $e^{-it\left|\frac{2}{h}\sin \frac{hk}{2}\right|^{\alpha}}$ is $\frac{2\pi}{h}$-periodic, we proceed without introducing yet another notation. An estimation on the Fourier side yields
\begin{align}\label{A1_estimate}
\text{RHS of }\eqref{A0_estimate} & \geq \| (e^{-it\left|\frac{2}{h}\sin \frac{hk}{2}\right|^{\alpha}} - e^{-it|k|^\alpha}) \widehat{u}_0\|_{l^2} - \| (e^{-it\left|\frac{2}{h}\sin \frac{hk}{2}\right|^{\alpha}} - e^{-it|k|^\alpha}) (\mathcal{F}[p_h d_h u_0] - \widehat{u}_0)\|_{l^2} \nonumber\\
&\geq \| (e^{-it\left|\frac{2}{h}\sin \frac{hk}{2}\right|^{\alpha}} - e^{-it|k|^\alpha}) \widehat{u}_0\|_{l^2} - C h^{\frac{\alpha}{2}} \| u_0 \|_{H^{\frac{\alpha}{2}}},
\end{align}
where $C>0$ is by \eqref{l4}. The lower bound of the phase difference is estimated. By direct computation,
\begin{equation*}
    \left|e^{-it\left|\frac{2}{h}\sin \frac{hk}{2}\right|^{\alpha}} - e^{-it|k|^\alpha}\right| = 2\left|\sin\left(\frac{2^\alpha t}{h^\alpha}\left(\left|\sin \frac{hk}{2}\right|^\alpha - \left|\frac{hk}{2}\right|^\alpha\right)\right)\right|.
\end{equation*}
Now we approximate $\left| \sin \frac{\theta}{2}\right|^{\alpha} - \left| \frac{\theta}{2} \right|^\alpha$ for $|\theta| \leq \frac{1}{10}$ where $\theta = hk$. Let $\theta > 0$ for the case $\theta < 0$ follows similarly. Since $\sin (\frac{\theta}{2}) = \frac{\theta}{2}\left(1 - \frac{\theta^2}{24} + \frac{\theta^4}{1920} - R\right)$ where $R = O(\theta^6)$, it follows from the Taylor expansion, in $\epsilon$, of $\sin(\frac{\theta}{2})^\alpha = (\frac{\theta}{2})^\alpha (1-\epsilon)^\alpha$ where $\epsilon = \frac{\theta^2}{24} - \frac{\theta^4}{1920} + R$ that
\begin{equation}\label{alternating}
    \left| \sin \frac{hk}{2}\right|^{\alpha} - \left| \frac{hk}{2} \right|^\alpha = \left|\frac{hk}{2}\right|^\alpha \left(-\frac{\alpha}{24}(hk)^2 + \frac{\alpha(5\alpha-2)}{5760} (hk)^4 - O_\alpha((hk)^4))\right),
\end{equation}
and therefore
\begin{equation}\label{upperbound}
    \frac{2^\alpha |t|}{h^\alpha}\cdot\left|\left|\sin \frac{hk}{2}\right|^\alpha - \left|\frac{hk}{2}\right|^\alpha \right| \leq \frac{\alpha |t|}{24} h^2 |k|^{2+\alpha}.
\end{equation}
To apply the lower bound estimate 
\begin{equation*}
    |\sin \theta| \geq \frac{2}{\pi} |\theta|,\ \text{for any\ }|\theta| \leq \frac{\pi}{2},
\end{equation*}
let $\theta = \frac{2^\alpha t}{h^\alpha}\left(\left|\sin \frac{hk}{2}\right|^\alpha - \left|\frac{hk}{2}\right|^\alpha\right)$ and by \eqref{upperbound}, assume that $E$ holds where
\begin{equation*}
    E = \{k \in \mathbb{Z}:|k| \leq \left(\frac{12 \pi}{\alpha}\right)^{\frac{1}{2+\alpha}}|t|^{-\frac{1}{2+\alpha}} h^{-\frac{2}{2+\alpha}}\}.
\end{equation*}
Then by the trigonometric lower bound estimate and \eqref{alternating},
\begin{equation*}
    \left|e^{-it\left|\frac{2}{h}\sin \frac{hk}{2}\right|^{\alpha}} - e^{-it|k|^\alpha}\right| \geq \frac{2^{2+\alpha} |t|}{\pi h^\alpha}\cdot\left|\left|\sin \frac{hk}{2}\right|^\alpha - \left|\frac{hk}{2}\right|^\alpha\right| \geq \frac{\alpha |t|}{12\pi} h^2 |k|^{2+\alpha},
\end{equation*}
and using this lower bound estimate, we have
\begin{equation}\label{low_frequency2}
    \| (e^{-it\left|\frac{2}{h}\sin \frac{hk}{2}\right|^{\alpha}} - e^{-it|k|^\alpha}) \widehat{u}_0\|_{l^2} \geq \| (e^{-it\left|\frac{2}{h}\sin \frac{hk}{2}\right|^{\alpha}} - e^{-it|k|^\alpha}) \widehat{u}_0\|_{l^2(E)} \gtrsim |t| h^2 \| |k|^{2+\alpha}\widehat{u}_0\|_{l^2(E)}.
\end{equation}
Let $k_0 = |t|^{-\frac{1}{2+\alpha}} h^{-\frac{2}{2+\alpha}}$. For $\epsilon > 0$ to be determined below, define
\begin{equation}\label{explicit_construction}
\widehat{u}_0^h(k) = \epsilon k_0^{-\frac{\alpha}{2}} \delta_{k_0}(k)    
\end{equation}
where $\delta_{k_0}$ is the Kronecker delta function supported at $k = k_0$. Hence by \eqref{sharpness_lb}, \eqref{reverse_triangle}, \eqref{A1_estimate}, and \eqref{low_frequency2},
\begin{equation*}
\begin{split}
\| p_h S_h(\cdot) d_h u_0^h - S(\cdot) u_0^h\|_{L^\infty_T L^2} &\gtrsim T h^2 \| |k|^{2+\alpha}\widehat{u}^h_0\|_{l^2(E)} - \| u_0^h \|_{H^{\frac{\alpha}{2}}}^3 T h^{\frac{\alpha}{2+\alpha}} - \| u_0^h \|_{H^{\frac{\alpha}{2}}} h^{\frac{\alpha}{2}}\\
&\simeq (\epsilon T^{\frac{\alpha/2}{2+\alpha}} - \epsilon^3 T) h^{\frac{\alpha}{2+\alpha}} - \epsilon h^{\frac{\alpha}{2}}\\
&\gtrsim (\frac{\epsilon}{2} - \epsilon^3) T^{\frac{\alpha/2}{2+\alpha}} h^{\frac{\alpha}{2+\alpha}},
\end{split}
\end{equation*}
where the last inequality assumes $T \leq 1$ and $h$ sufficiently small depending on $\alpha$ and $T$; for example, $h \leq \left(\frac{T^{\frac{\alpha/2}{2+\alpha}}}{2}\right)^{\frac{4+2\alpha}{\alpha^2}}$ would do. The proof is complete by taking $0 < \epsilon < \frac{1}{\sqrt{2}}$.
\end{proof}

\begin{corollary}
Let $T>0$ be as \Cref{sharpness} and suppose
\begin{equation}\label{sharpness_cor}
    \sup_{\| u_0 \|_{H^{\frac{\alpha}{2}}(\mathbb{T}) }< R}\| p_h S_h(\cdot) d_h u_0 - S(\cdot) u_0\|_{L^\infty_T L^2} \lesssim_{T,R,\alpha} h^{p},
\end{equation}
for some $R>0,\ p>0$. Then $\max \{p:\eqref{sharpness_cor}\ \text{holds for all\ } h > 0\} = \frac{2}{2+\alpha}$.
\end{corollary}
\begin{proof}
    By \Cref{result2}, $\frac{2}{2+\alpha}$ satisfies $\eqref{sharpness_cor}$ for any $R > 0$. Any $p > \frac{2}{2+\alpha}$ is not in the desired set by an explicit construction in \Cref{sharpness}.
\end{proof}
\begin{remark}
    The sharpness of the convergence rate is expected to be $O(h^{\frac{2s}{2+\alpha}})$ in $H^s(\mathbb{T})$ for $s < \frac{\alpha}{2}$, i.e., for data of infinite energy. The proof in this regime, given the technical difficulty due to the absence of the Sobolev embedding, is left for further research. 
\end{remark}

We have shown that the convergence rate of the numerical scheme given by \eqref{continuum_local} or \eqref{continuum_global} is sublinear at worst for general Sobolev data. In numerical computations using softwares, the high frequency components of $u_0$ are often truncated, and therefore the Fourier support of $\widehat{u}_0$ is assumed to be compact. To motivate further discussion on the relationship between numerical convergence and the compactness of Fourier support, see \Cref{appendix_exact} for concrete examples of exact solutions also considered in \cite{burq2002instability}. In the following proposition, the sharp linear convergence illustrated by the example \eqref{exactsolution} is generalized. More remains to be studied on the nonlinear evolution of Fourier modes on the lattice.

\begin{proposition}\label{compact_support}
    For $\alpha \in (1,2]$, assume $u_0 \in H^{\frac{\alpha}{2}}(\mathbb{T})$ and $k_{m} := \max\{|k|: k \in supp(\widehat{u}_0)\}<\infty$. Suppose there exist $T>0,\ k_c > 0$, and $h_0 > 0$ such that $supp(\mathcal{F}_h[u_h(t)]) \subseteq [-k_c,k_c]$ for all $|t| \leq T,\ h < h_0$. Then,    \begin{equation}\label{compact_support2}
        \| p_h S_h(t) d_h u_0 - S(t) u_0\|_{L^2(\mathbb{T})} \leq C h (k_{max})^{1-\frac{\alpha}{2}},
    \end{equation}
    where $k_{max} = \max (k_m,k_c)$ and $C = C(\| u_0 \|_{H^{\frac{\alpha}{2}}},\alpha) > 0$, for $t \in [0,T]$ and $h>0$ sufficiently small given by \eqref{h_bound}.
\end{proposition}
\begin{proof}
The argument proceeds as in the proof of \Cref{result} where we may assume $T>0$ is at most the time of existence stated in \Cref{result}. Take
\begin{equation}\label{h_bound}
    h < C_0 \min(h_0,(3k_{max})^{-1},(T k_{max}^{1+\alpha})^{-1},(T^2 k_{max}^{1+\alpha})^{-1}),
\end{equation}
where $C_0(\| u_0 \|_{H^{\frac{\alpha}{2}}},\alpha)>0$ to be determined. Since $2k_m < \frac{2\pi}{h}$, the period of $\mathbb{T}_h^*$, we have $supp(\mathcal{F}_h [d_h u_0]) = supp(\widehat{u}_0)$ by \eqref{dft}. Hence by the triangle inequality and \eqref{l4},
\begin{align}
\| p_h U_h(t) d_h u_0 - U(t) u_0\|_{L^2} &\leq \| p_h U_h(t) d_h u_0 - U(t) p_h d_h u_0\|_{L^2} + \| U(t)(p_h d_h u_0 - u_0)\|_{L^2} \nonumber\\
&\lesssim \| (e^{-it\left|\frac{2}{h}\sin \frac{hk}{2}\right|^{\alpha}} - e^{-it|k|^\alpha}) P_h(k)\mathcal{F}_h [d_h u_0]\|_{l^2_{\{|k| \leq h^{-1}\}}} + h \| u_0 \|_{H^1} \nonumber\\
&\lesssim \alpha |t| h^2 \| |k|^{2+\alpha} \mathcal{F}_h [d_h u_0]\|_{l^2_{\{|k| \leq h^{-1}\}}} + h \| u_0 \|_{H^1} \label{low_frequency3}\\
&\lesssim \| u_0 \|_{H^{\frac{\alpha}{2}}}\left(\alpha T k_m^{2+\frac{\alpha}{2}}h^2 + k_m^{1-\frac{\alpha}{2}}h\right) \lesssim \| u_0 \|_{H^{\frac{\alpha}{2}}} k_m^{1-\frac{\alpha}{2}} h, \label{linear_estimate}
\end{align}
where the third inequality estimates the phase difference as \eqref{upperbound} for $|k| \leq h^{-1}$ and the last inequality follows from $h < (\alpha T k_m^{1+\alpha})^{-1}$. Then the nonlinear terms are estimated. Since $supp(\mathcal{F}_h[|u_h|^2 u_h]) \subseteq [-3k_c,3k_c]$,
\begin{align}
\int_0^t \| (p_h U_h(t-\tau) - U(t-\tau)p_h)(|u_h|^2 u_h)(\tau)\|_{L^2} d\tau &\lesssim h^2 \int_0^t |t-\tau| \cdot \| |k|^{2+\alpha} P_h(k) \mathcal{F}_h [|u_h|^2 u_h]\|_{l^2_{\{|k| \leq h^{-1}\}}} d\tau \nonumber \\
&\lesssim h^2 (3k_c)^{2+\frac{\alpha}{2}} \int_0^t |t-\tau| \cdot \| |u_h(\tau)|^2 u_h(\tau) \|_{H_h^{\frac{\alpha}{2}}} d\tau \nonumber \\
&\lesssim C T^2 h^2 (3k_c)^{2+\frac{\alpha}{2}}, \label{linear_estimate2} 
\end{align}
where the first inequality is estimated as \eqref{low_frequency3} and the last inequality of $H_h^{\frac{\alpha}{2}}$ is by the Sobolev algebra property with $C$ depending only on the mass and energy of $u_0$ by the discrete Gagliardo-Nirenberg inequality. Another nonlinear term is estimated by \eqref{l3}.
\begin{equation}\label{commutator_estimate}
    \int_0^t \| p_h(|u_h|^{2}u_h) - |p_h u_h|^{2}p_h u_h\|_{L^2} d\tau \lesssim h \int_0^t \| u_h \|_{L^\infty_h}^2 \| u_h \|_{H^1_h} d\tau \lesssim \| u_h \|_{H_h^{\frac{\alpha}{2}}}^3 T h k_c^{1-\frac{\alpha}{2}}.  
\end{equation}
Combining \eqref{linear_estimate}, \eqref{linear_estimate2}, \eqref{commutator_estimate}, we have
\begin{equation*}
\begin{split}
\| p_h u_h(t) - u(t)\|_{L^2} &\lesssim h k_m^{1-\frac{\alpha}{2}}  + T^2 h^2 (3k_c)^{2+\frac{\alpha}{2}} + T h k_c^{1-\frac{\alpha}{2}} + \int_0^t (\| p_h u_h \|_{L^\infty}^2 + \| u \|_{L^\infty}^2) \| p_h u_h - u \|_{L^2} d\tau\\
&\lesssim (1+T) h (k_{max})^{1-\frac{\alpha}{2}} + \int_0^t (\| p_h u_h \|_{L^\infty}^2 + \| u \|_{L^\infty}^2) \| p_h u_h - u \|_{L^2} d\tau. 
\end{split}
\end{equation*}
That $\| p_h u_h \|_{L^2_T L^\infty}^2 + \| u \|_{L^2_T L^\infty}^2 = O(1)$ can be shown as in the proof of \Cref{result}, and therefore \eqref{compact_support2} holds by the Gronwall's inequality.
\end{proof}

\begin{remark}
The computation of numerical error could take place in different function spaces with different interpolation methods. While \Cref{compact_support} gives linear convergence in $L^2(\mathbb{T})$ for solutions with compact Fourier support, the faster quadratic convergence is not expected for linearly interpolated solutions. Indeed the second derivative acting on $p_h$ yields the Dirac delta functions, which are not square-integrable. Observe, however, that the exact solution \eqref{exactsolution} with the initial datum $u_0(x) = A |n|^{-s} e^{inx}$ converges quadratically in $L^2_h$ as can be shown explicitly as
\begin{align}\label{exact_discrete}
    \| u_h(t) - d_hu(t)\|_{L^2_h} &= \sqrt{2\pi} |A| |n|^{-s} \left|\frac{e^{ihn}-1}{ihn}\right|\cdot \left|e^{-it(|\frac{2}{h}\sin \frac{hn}{2}|^\alpha + \mu |A|^2 |n|^{-2s}|\frac{e^{ihn}-1}{ihn}|^2} - e^{-it(|n|^\alpha + \mu |A|^2 |n|^{-2s}}\right|\nonumber\\
    &= \left(\frac{\sqrt{2\pi}}{24} |A t n|^{2-3s} \cdot \left|\alpha |n|^{\alpha + 2s} + 2 \mu |A|^2\right|\right) h^2 + O(h^3).
\end{align}
\end{remark}

\section{Modulational instability of CW solutions.}\label{MI_section}

Whereas the previous sections studied the time evolution of arbitrary Sobolev data, this section investigates that of CW solutions. An analysis of MI of CW solution under \eqref{fdnls} is given. The regions of linear stability and instability are given analytically, and an explicit derivation of the gain spectrum is computed. \Cref{linstability_alpha_low,chaos,maxgain} are generated based on FFT while \Cref{instability_region} is due to Mathematica.
\begin{figure}[H]
\begin{subfigure}[h]{0.29\linewidth}
\includegraphics[width=\linewidth]{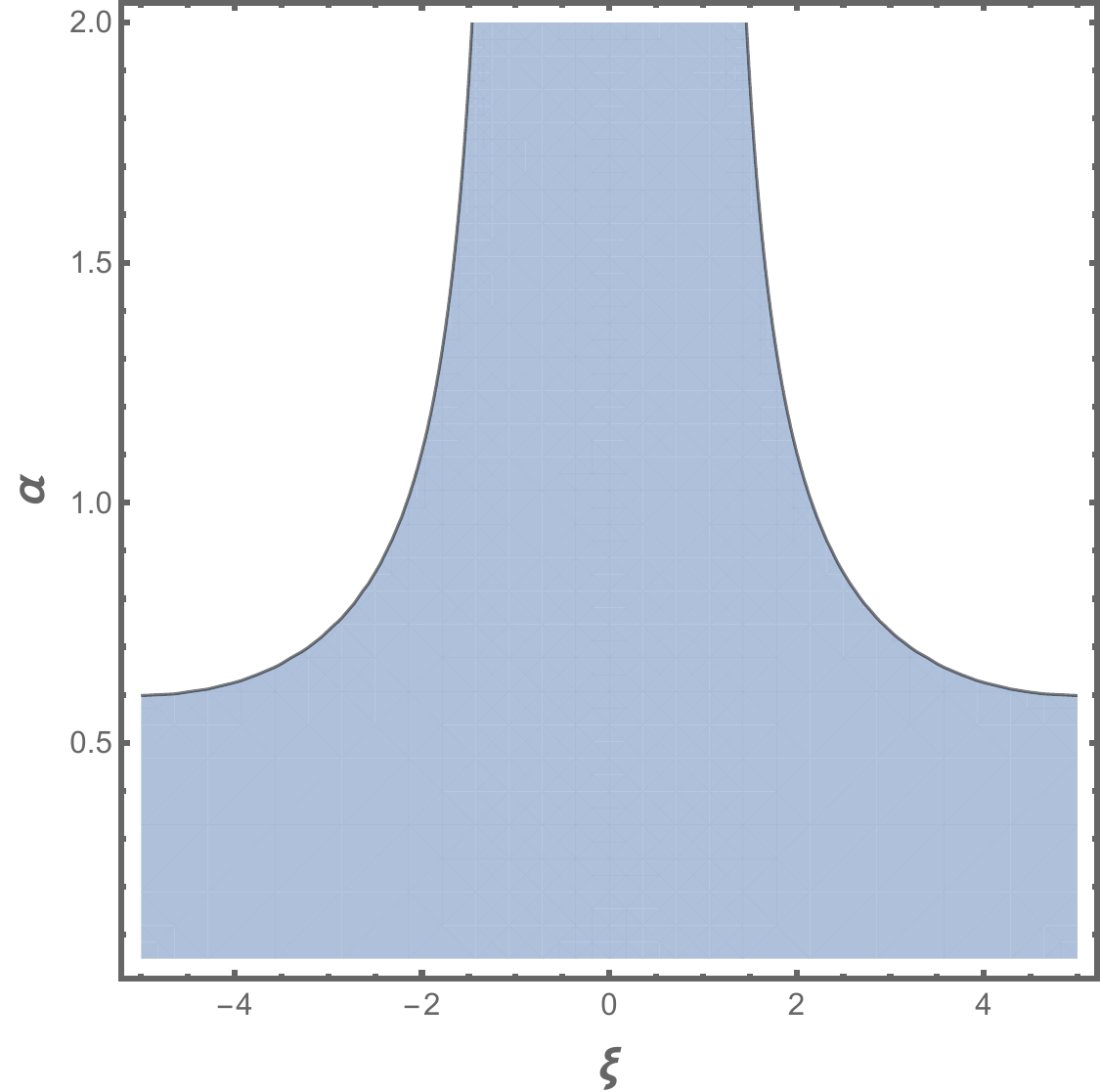}
\caption{$A = 1$}
\end{subfigure}
\hfill
\begin{subfigure}[h]{0.29\linewidth}
\includegraphics[width=\linewidth]{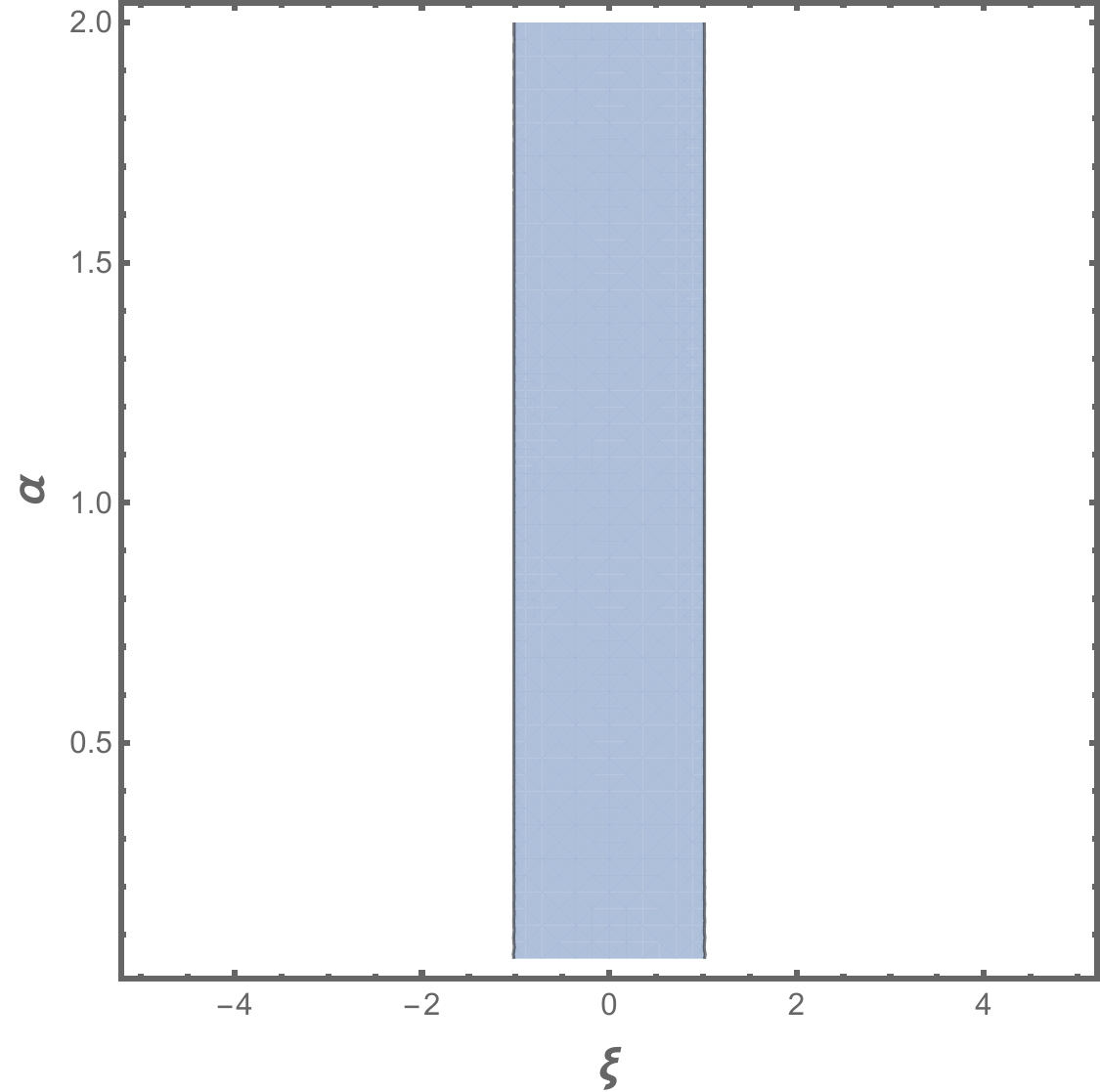}
\caption{$A = \frac{1}{\sqrt{2}}$}
\end{subfigure}
\hfill
\begin{subfigure}[h]{0.29\linewidth}
\includegraphics[width=\linewidth]{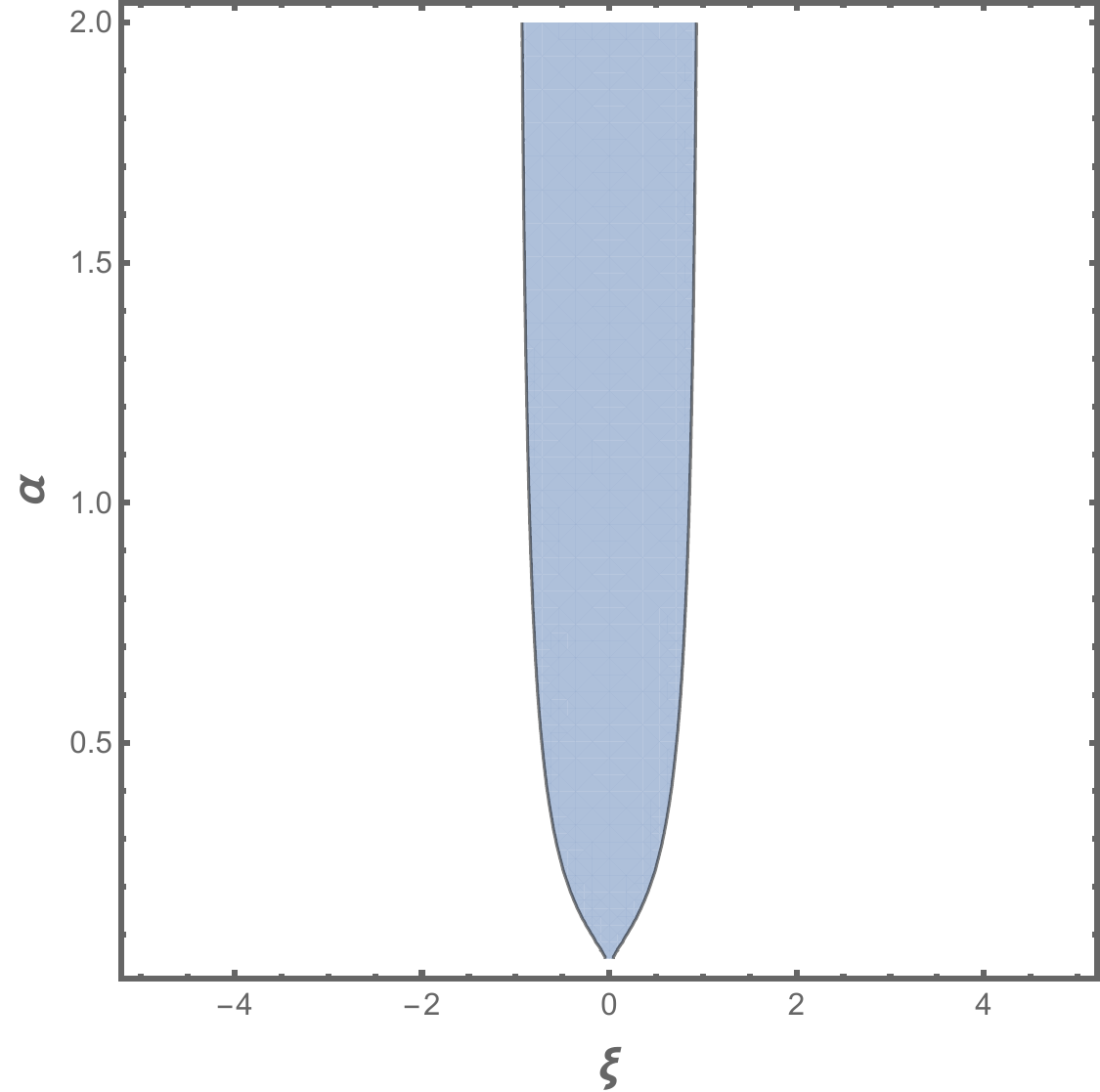}
\caption{$A = 0.5$}
\end{subfigure}
\begin{subfigure}[h]{0.29\linewidth}
\includegraphics[width=\linewidth]{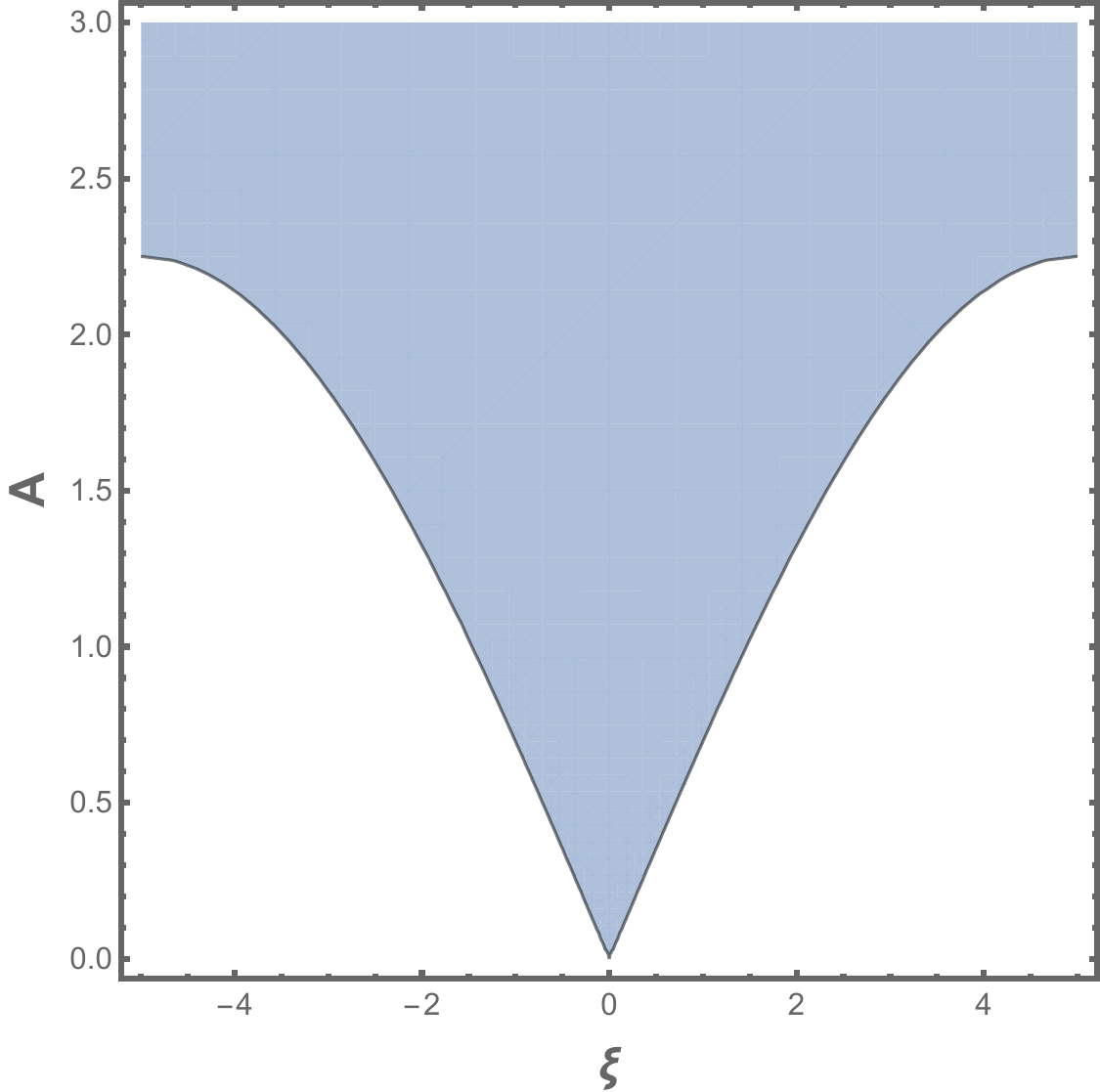}
\caption{$\alpha = 2$}
\end{subfigure}
\hfill
\begin{subfigure}[h]{0.29\linewidth}
\includegraphics[width=\linewidth]{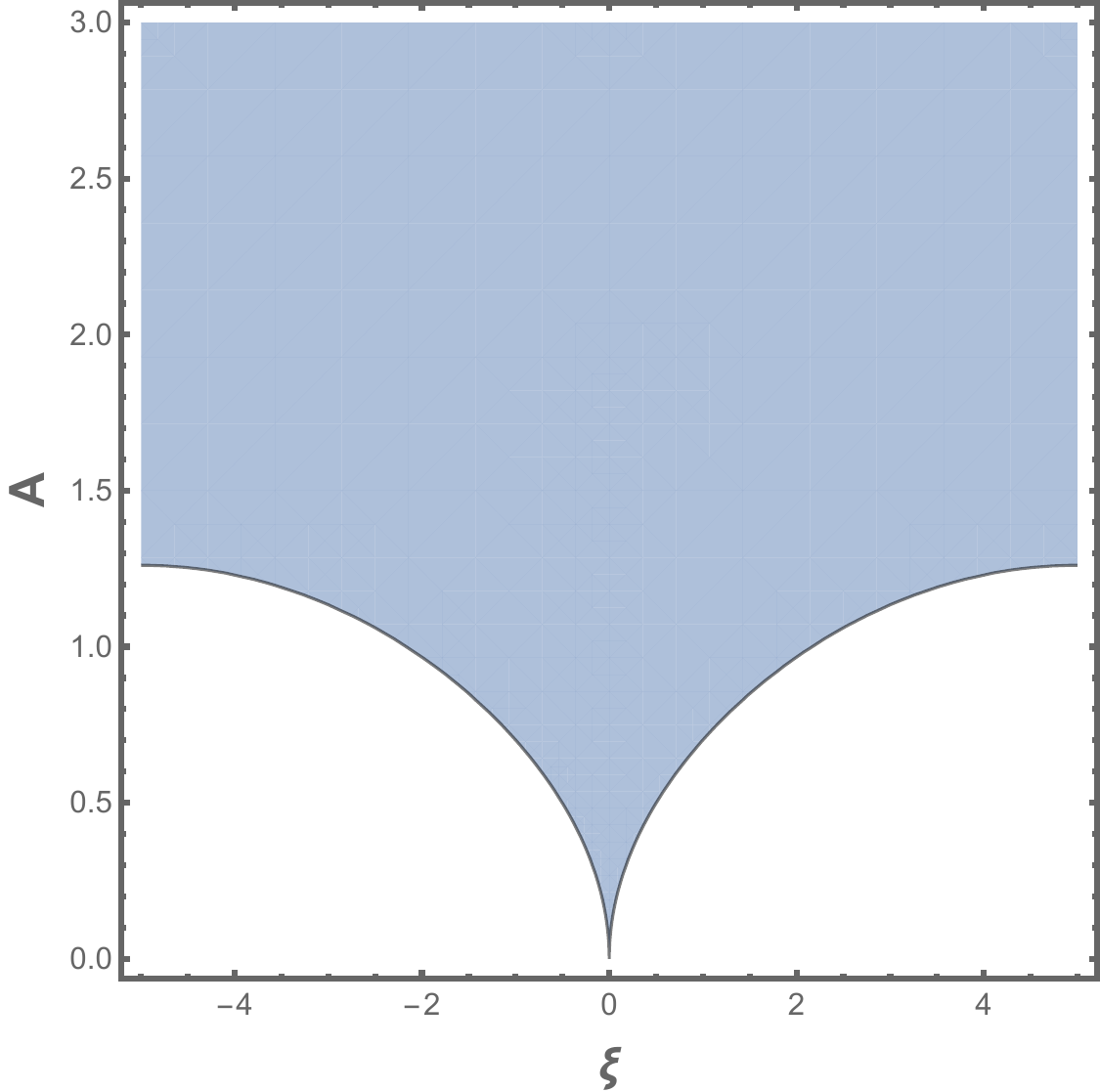}
\caption{$\alpha = 1$}
\end{subfigure}
\hfill
\begin{subfigure}[h]{0.29\linewidth}
\includegraphics[width=\linewidth]{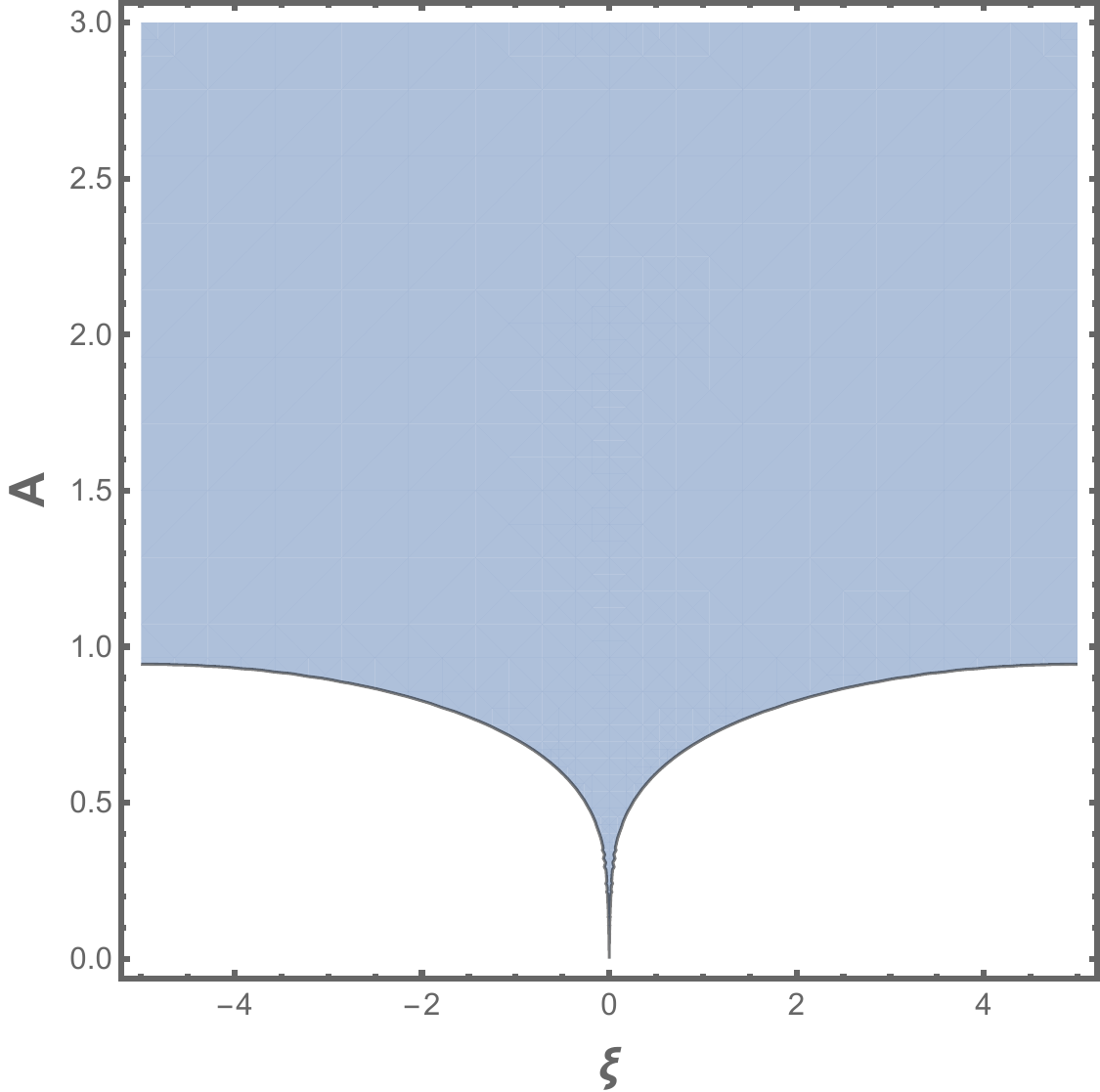}
\caption{$\alpha = 0.5$}
\end{subfigure}
\caption{The region of linear instability given by \eqref{region_instability} (in blue) is plotted in $(\xi,A)$ and $(\xi,\alpha)$ for $h = \frac{\pi}{5}$.}
\label{instability_region}
\end{figure}
Define $u_{h}^{cw} = Ae^{-i \mu |A|^2 t}$ and let $u_h = (A + \epsilon v_h (x,t))e^{-i \mu |A|^2 t}$ where $v_h(x,t) \in \mathbb{C},\ |\epsilon| \ll 1$, and $A \in \mathbb{R}$ without loss of generality. The $O(\epsilon)$ term yields
\begin{equation*}
    i \frac{d v_h}{dt} = (-\Delta_h)^{\frac{\alpha}{2}} v_h + \mu A^2 (v_h + \overline{v}_h).
\end{equation*}
Taking the real and imaginary parts, i.e. $v_h = f_h + i g_h$, we have 
\begin{equation}\label{perturbation2}
\frac{d}{dt}
\begin{pmatrix}
    f_h \\ g_h
\end{pmatrix}
=
\begin{pmatrix}
    0 &  (-\Delta_h)^{\frac{\alpha}{2}}\\
    -(-\Delta_h)^{\frac{\alpha}{2}} -2 \mu A^2 & 0
\end{pmatrix}
\begin{pmatrix}
    f_h \\ g_h
\end{pmatrix}.
\end{equation}
Taking the discrete Fourier transform both sides and the ansatz $\mathcal{F}_h[f_h] = P_h(k)e^{-i \Omega t},\ \mathcal{F}_h[g_h] = G_h(k)e^{-i \Omega t}$, \eqref{perturbation2} becomes an eigenvalue problem whose nontrivial solution $(P_h,G_h)$ exists if and only if $\Omega$ satisfies the dispersion relation given by
\begin{equation}\label{dispersion_relation}
    \Omega^2(k) = \left| \frac{2}{h} \sin (\frac{hk}{2}) \right|^{\alpha} \left(\left| \frac{2}{h} \sin (\frac{hk}{2}) \right|^{\alpha} + 2 \mu A^2\right).
\end{equation}

When $\mu = 1$, the system is linearly stable since $\Omega^2 \geq 0$ and henceforth assume $\mu = -1$. The region of linear instability and the corresponding gain spectrum are given by
\begin{align}
\left| \frac{2}{h} \sin (\frac{hk}{2}) \right|^{\alpha} &< 2A^2,\ |k| \leq \frac{\pi}{h},\label{region_instability} \\ 
G(\xi,A,\alpha,h) &:= \sqrt{\left| \frac{2}{h} \sin (\frac{h\xi}{2}) \right|^{\alpha} \left(2A^2 - \left| \frac{2}{h} \sin (\frac{h\xi}{2}) \right|^{\alpha}\right)}, \nonumber
\end{align}
where we denote $\xi \in \mathbb{R},\ k \in \mathbb{Z}$.

See \Cref{instability_region} that illustrates \eqref{region_instability}. Top row: when $A = \frac{1}{\sqrt{2}}$, the region is independent of $\alpha$. If $A \ll 1$ and $\alpha \ll 1$, then there exist no $k \in \mathbb{Z}\setminus \{0\}$ that satisfies \eqref{region_instability}, i.e., linear stability. On the other hand, for $A > \frac{1}{\sqrt{2}}$ and $\alpha \ll 1$, any $k \in \mathbb{T}_h$ satisfies \eqref{region_instability}, i.e., linear instability. Botton row: when $\alpha = 2$, $A(\xi)$ behaves as a kink and when $\alpha < 2$, $A(\xi)$ behaves as a cusp near $\xi = 0$. More precisely, $A \sim_h \frac{|\xi|^{\frac{\alpha}{2}}}{\sqrt{2}}$ as $\xi \rightarrow 0$. Note that the region approaches $\{|\xi|^\alpha < 2A^2\}$ as $h \rightarrow 0$. Hence for a fixed $h>0$,  the system is linearly unstable in the entire bandwidth if $|A|$ is sufficiently large, i.e., if $2A^2 > (\frac{2}{h})^\alpha$.

\vspace*{-3cm}
\begin{figure}[H]
\begin{subfigure}[b]{0.50\linewidth}
\includegraphics[width=\linewidth]{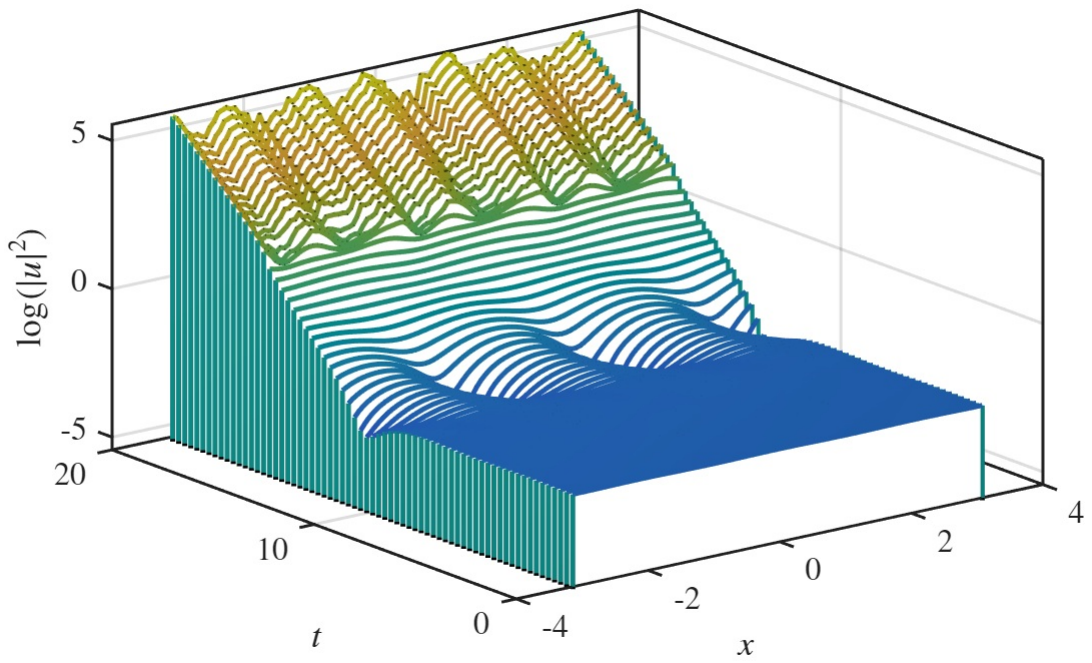}
\vspace*{-3.5cm}
\caption{$(\alpha,A) = (0.25,0.1)$}
\end{subfigure}
\hfill
\begin{subfigure}[b]{0.50\linewidth}
\includegraphics[width=\linewidth]{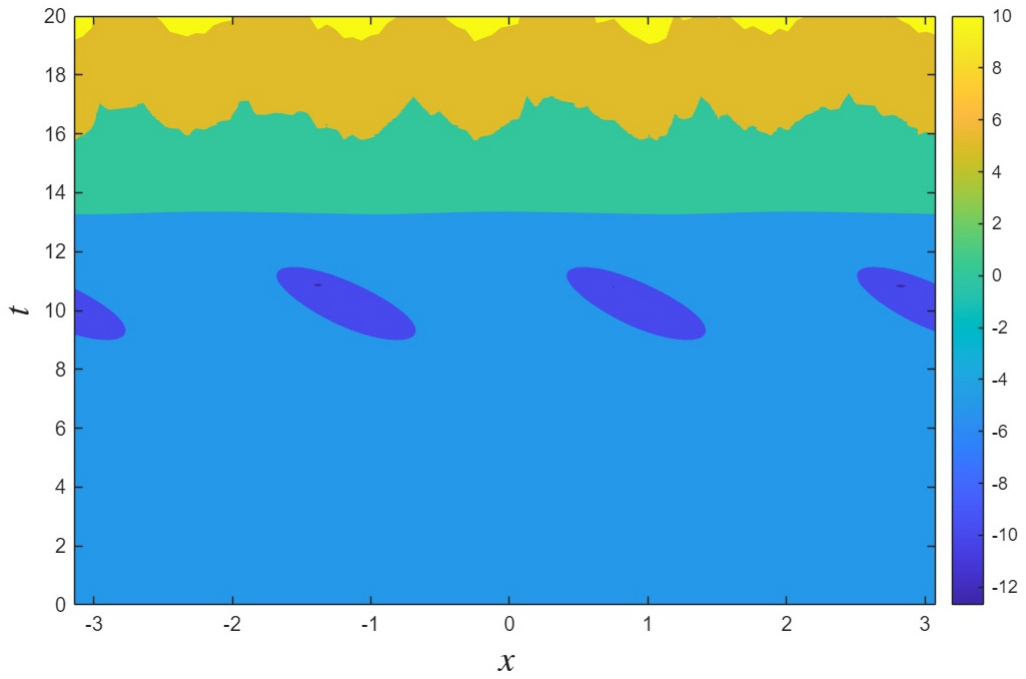}
\vspace*{-3.5cm}
\caption{$(\alpha,A) = (0.25,0.1)$}
\end{subfigure}

\begin{subfigure}[b]{0.50\linewidth}
\vspace*{-3cm}
\includegraphics[width=\linewidth]{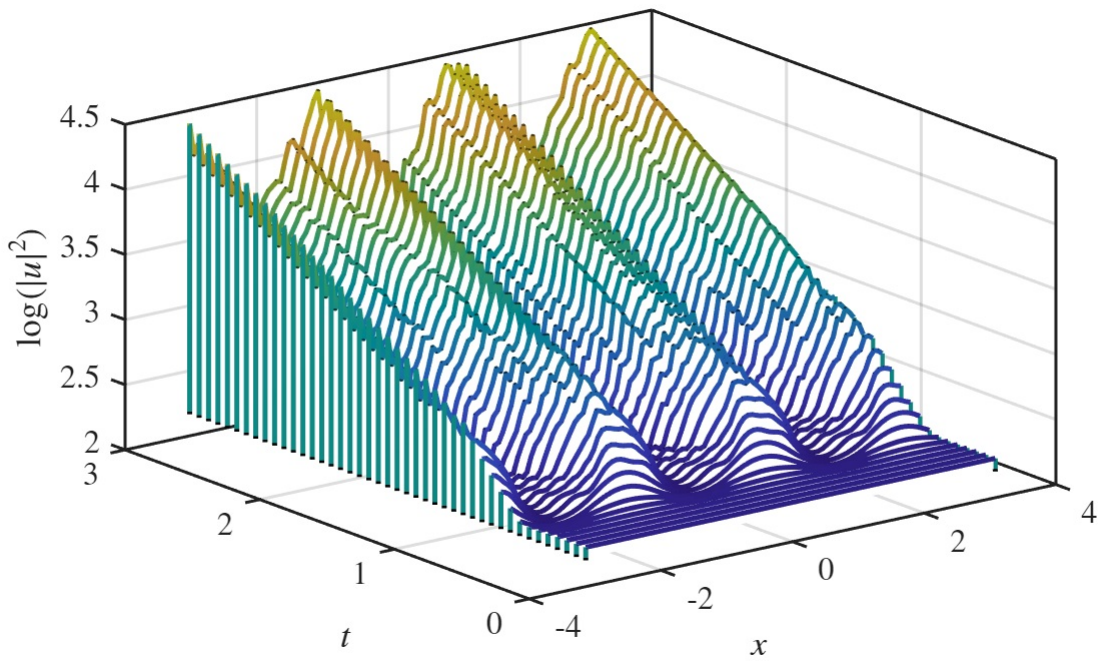}
\vspace*{-3.5cm}
\caption{$(\alpha,A) = (0.25,10)$}
\end{subfigure}
\hfill
\begin{subfigure}[b]{0.50\linewidth}
\vspace*{-3cm}
\includegraphics[width=\linewidth]{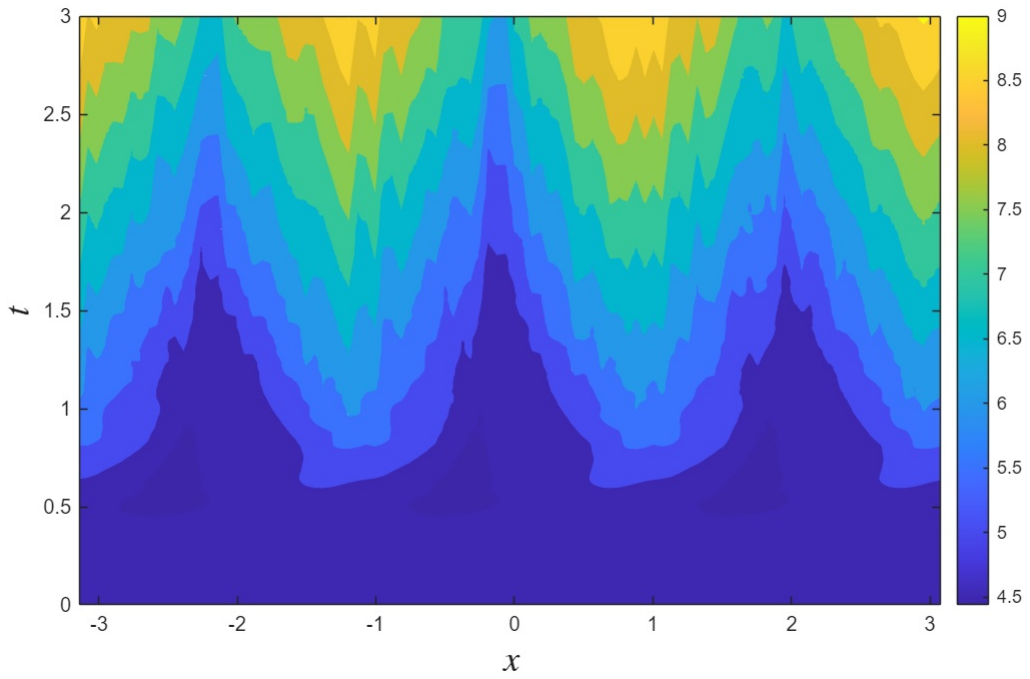}
\vspace*{-3.5cm}
\caption{$(\alpha,A) = (0.25,10)$}
\end{subfigure}

\caption{The log plots used $(\alpha,h,k) = (0.25,\frac{\pi}{50},3)$ and $u_0(x) = A + 10^{-3} e^{ikx}$ where the right plot is the contour of the left plot.}
\label{linstability_alpha_low}
\end{figure}

If \eqref{region_instability} holds, then the maximum exponential gain $\Omega_m$ that occurs at $k_m$, the fastest-growth frequency, can be computed explicitly by computing the derivative of \eqref{dispersion_relation} treating $k$ as real. By direct computation, for $\xi \in (0,\frac{\pi}{h})$,
    \begin{equation*}
        \frac{d}{d\xi} \Omega^2 (\xi) = \alpha h \cot(\frac{h\xi}{2}) \left( \frac{2}{h} \sin (\frac{h\xi}{2}) \right)^{\alpha}\left(\left( \frac{2}{h} \sin (\frac{h\xi}{2}) \right)^{\alpha}-A^2\right).
    \end{equation*}  
  
    If $(\frac{2}{h})^{\alpha} \leq A^2$, then $k_m = \pm M$ and
    \begin{equation}\label{gain}
        \Omega_m = \sqrt{\left(2A^2 - \left(\frac{2}{h}\right)^\alpha\right)\left(\frac{2}{h}\right)^\alpha}.
    \end{equation}
    
    If $(\frac{2}{h})^{\alpha} > A^2$, let $\xi_m \in (0,\frac{\pi}{h})$ be real such that $\left(\frac{2}{h}\sin(\frac{h\xi_m}{2})\right)^\alpha = A^2$, or equivalently, $\xi_m = \frac{2}{h} \sin^{-1} \left(\frac{h |A|^{\frac{2}{\alpha}}}{2}\right)$. It can be verified directly that $\pm\xi_m$ is the unique frequency that maximizes $-\Omega^2$. Therefore $|k_m| \in \{\lfloor \xi_m \rfloor, \lceil \xi_m \rceil\}$ and $\Omega_m = \sqrt{-\Omega^2(k_m)}$. Observe that $\Omega_m^\prime := \sqrt{-\Omega^2(\xi_m)} = A^2$, independent of $h,\alpha$. A summary of notable results follows.     
\begin{itemize}
    \item In the continuum limit, the region of instability is $\{|k|^\alpha < 2 A^2\}$. The region of linear instability for fDNLS strictly contains that of fNLS since $|\sin(z)| < |z|$ when $0 < |z| \leq \frac{\pi}{2}$. 
    \item The system is linearly stable if $|A| \ll 1,\ \alpha \ll 1$, which is in stark contrast to the system posed on $h\mathbb{Z}$ where given any $A>0$, there exists a real $k$ sufficiently small that satisfies \eqref{region_instability}. In \Cref{linstability_alpha_low}, the solution is linearly stable when $|A| \ll 1$. However nonlinearity begins to dominate from $t = 10$ with the emergence of $k$ troughs where $k=3$ was used in \Cref{linstability_alpha_low}. Indeed numerical experiments suggest the perturbation of $\epsilon e^{ikx}$, with $|\epsilon| \ll 1$ and $k \in \mathbb{T}_h^*$, triggers the emergence of $k$ troughs as the linear stability is supplanted by highly nonlinear wave evolution. For high $A$ values, the system is linearly unstable.
\end{itemize} 
\vspace{-18ex}

\begin{figure}[H]
\begin{subfigure}[b]{0.50\linewidth}
\includegraphics[width=\linewidth]{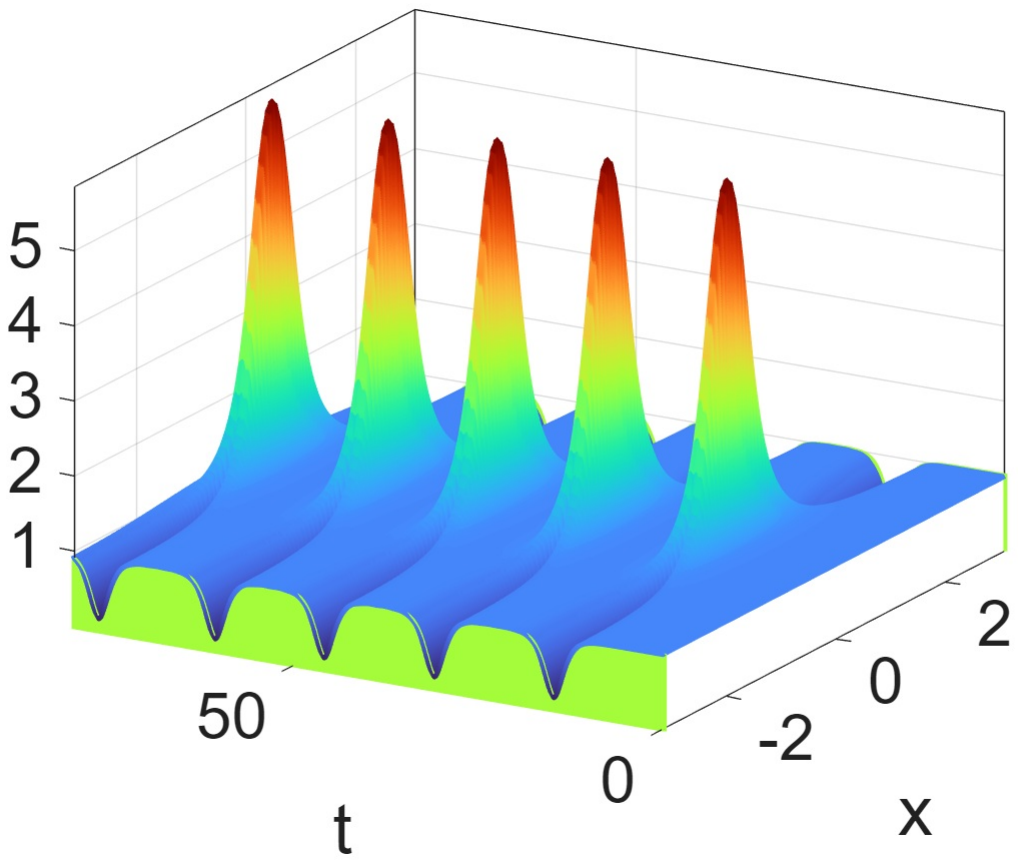}
\vspace*{-2.5cm}
\caption{$\alpha = 2.0$}
\end{subfigure}
\hfill
\begin{subfigure}[b]{0.50\linewidth}
\includegraphics[width=\linewidth]{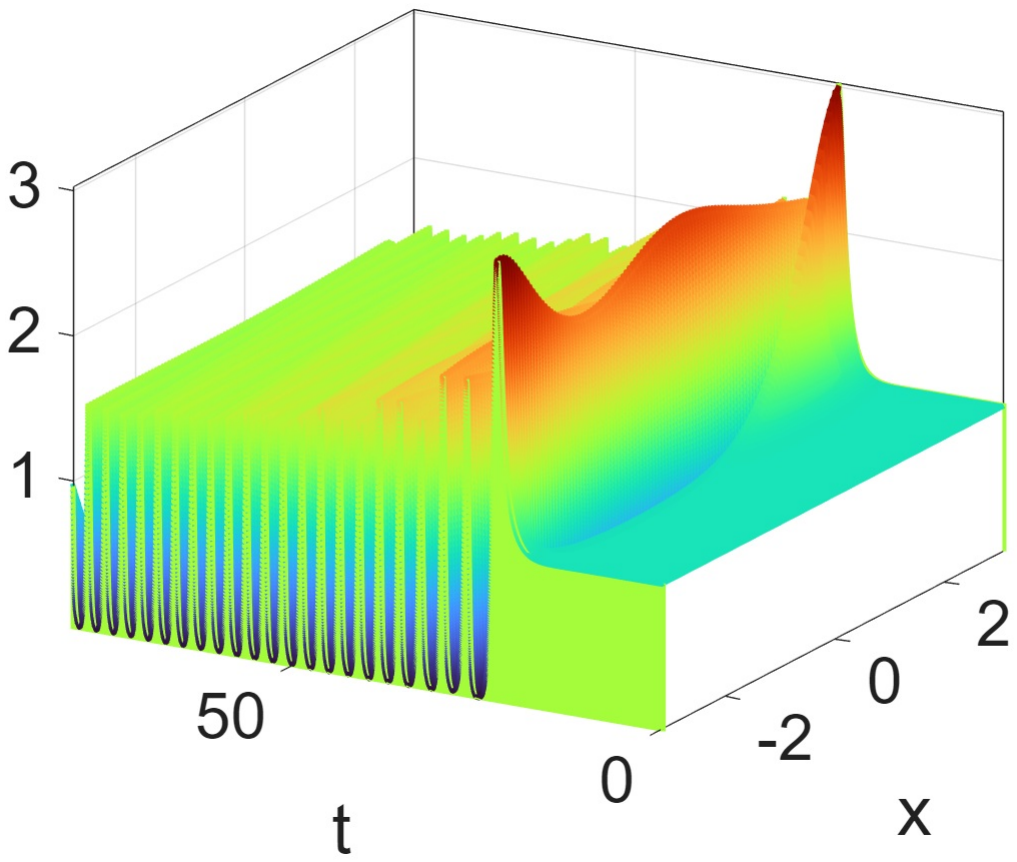}
\vspace*{-2.5cm}
\caption{$\alpha = 1.7$}
\end{subfigure}

\begin{subfigure}[b]{0.50\linewidth}
\vspace*{-2cm}
\includegraphics[width=\linewidth]{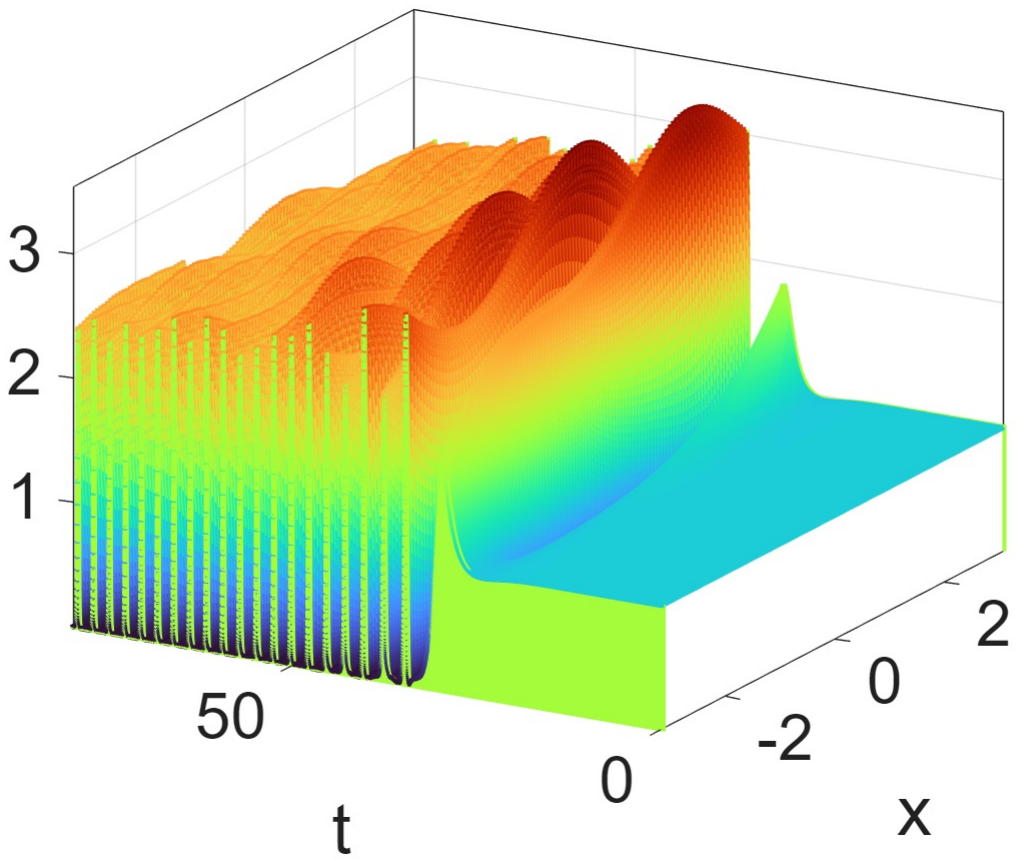}
\vspace*{-2.5cm}
\caption{$\alpha = 1.4$}
\end{subfigure}
\hfill
\begin{subfigure}[b]{0.50\linewidth}
\vspace*{-2cm}
\includegraphics[width=\linewidth]{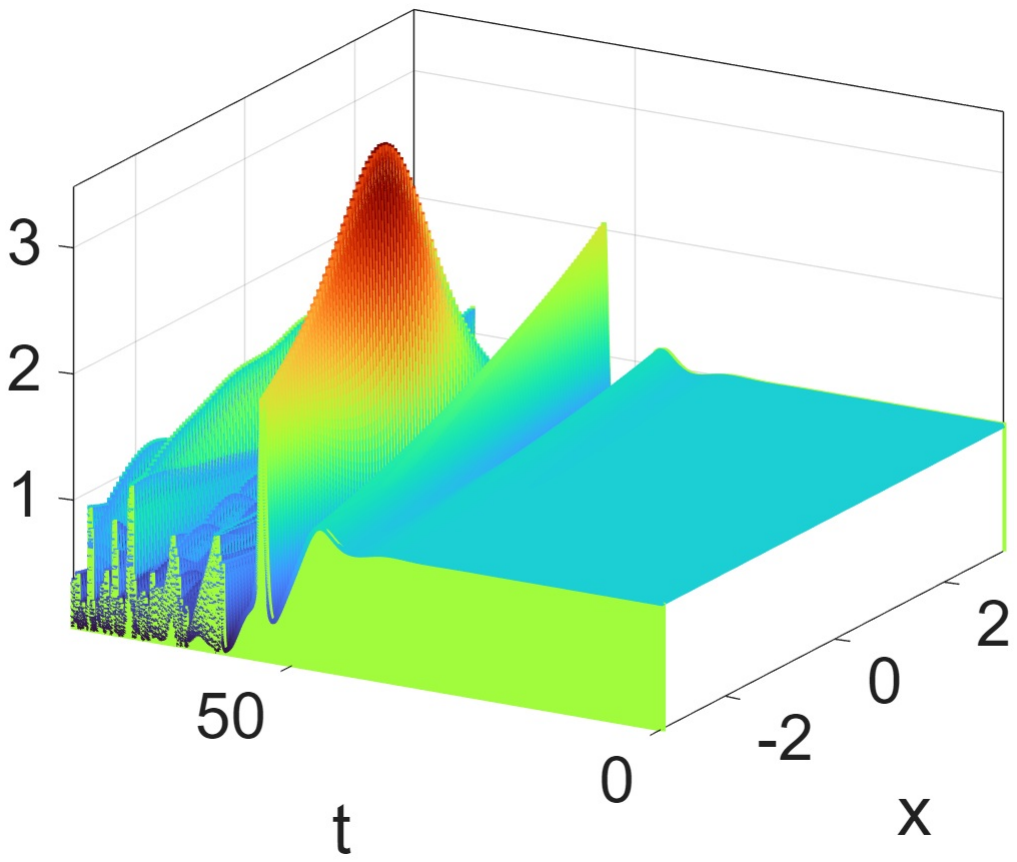}
\vspace*{-2.5cm}
\caption{$\alpha = 1.1$}
\end{subfigure}

\caption{Plots of $|u_h(x,t)|^2$ with $h = \frac{\pi}{50}$ and $u_0(x) = 1 + 10^{-6}(e^{ix}+e^{-ix})$.}
\label{chaos}
\end{figure}

\begin{itemize}
    \item A transition into chaos as $\alpha$ decreases is illustrated in \Cref{chaos}, consistent with \cite{korabel2007transition}. For $\alpha = 2$, the recurrence of localization was observed as expected; see \cite{duo2021dynamics} for a detailed numerical study on the nonlinear evolution of fNLS using the split-step Fourier spectral method. As $\alpha$ decreases to $1$, such clear recurrence was not observed with the development of irregular amplitudes. The time of first localization was observed to be delayed as $\alpha \rightarrow 1+$.
    \item By \eqref{gain}, $\Omega_m$ grows linearly in $|A|$ asymptotically as $|A| \rightarrow \infty$ when $h^{-\alpha} \ll A^2$. The transition occurs when $h^{-\alpha} \simeq A^2$. When $h^{-\alpha} \gg A^2$, we have $\Omega_m^\prime = \sqrt{-\Omega^2(\xi_m)} = A^2$; recall that $\xi_m$ may not be an integer and that $k_m = \lfloor \xi_m \rfloor$ or $\lceil \xi_m \rceil$. \Cref{maxgain} reports, for multiple L\'evy indices, the initial quadratic growth of $\Omega_m$ for sufficiently small $A$, followed by a non-quadratic behavior. Linear stability analysis suggests that the linear growth should follow, consistent with our numerical experiments when $10 \leq A \leq 20$. However for larger values of $A$, the spectrum of instability for higher harmonics is larger, and therefore the nonlinear evolution seems to be non-negligible. 
\end{itemize}

\begin{figure}[H]
\vspace{-47ex}
\includegraphics[width=1.0\linewidth]{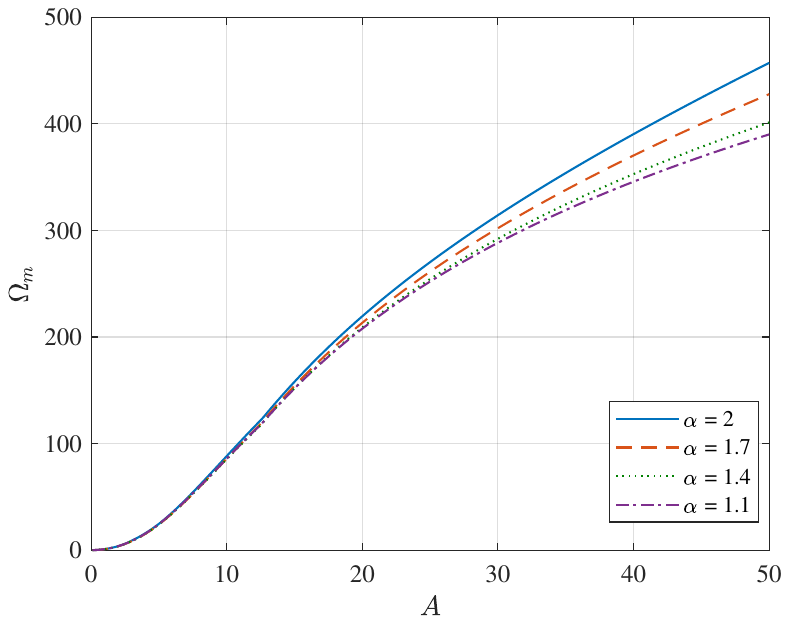}
\vspace{-47ex}
\caption{Parameters used for the plot: $h = \frac{\pi}{50},\ u_0(x) = A + 10^{-5}e^{50 i x}$. The quadratic growth of the maximum gain for $|A| \ll 1$ (continuum regime) deviates as the amplitude increases.}
\label{maxgain}
\end{figure}
  
\section{Conclusion.} 
Motivated by recent trends in fractional calculus and nonlocal dynamics, we investigated fDNLS on a periodic lattice. The continuum limit for data below the energy space was shown. However the method of periodic discrete Strichartz estimates was insufficient to establish the desired convergence up to $s_{fNLS} = \frac{2-\alpha}{4}$. The convergent dynamics of fDNLS was illustrated in the context of fractional MI of CW solutions. It was shown that the nonlocal parameter $\alpha$ triggers a broader spectrum of higher mode excitations if $|A| > \frac{1}{\sqrt{2}}$ while the spectrum shrinks if $|A| < \frac{1}{\sqrt{2}}$. The dependence of the maximum gain on $A,h,\alpha$ was shown analytically and numerically, demonstrated by the emergence of chaos as $\alpha$ departs from $\alpha = 2$ where the long-range coupling yields strong correlation between two distant lattice sites.

\appendix
\section{Appendix: well-posedness and uniform estimates.}\label{Uniform}

The well-posedness results of \eqref{fdnls}, \eqref{fnls} are given, followed by the uniform estimates needed to establish the continuum limit.

The quantitative measure of dispersive smoothing can be obtained by averaging over space and time variables under the unitary evolution. Recall that the Bourgain norm measures the $L^2$ norm of the space-time Fourier transform weighted by the deviation from the hypersurface defined as the zero-set of the dispersion relation. Let $s,b \in \mathbb{R}$ and $\widehat{f}(k,\tau) = \int_\mathbb{R}\int_{\mathbb{T}} f(x,t)e^{-i(kx + \tau t)}dx dt$. Define
\begin{equation*}
    \| f \|_{X^{s,b}}^2 = \int_{\mathbb{R}}\sum\limits_{k \in \mathbb{Z}}\langle k \rangle^{2s} \langle \tau - |k|^\alpha\rangle^{2b}|\widehat{f}(k,\tau)|^2 d\tau.
\end{equation*}
To establish local well-posedness in $[0,T]$, consider $\mathcal{C}_f = \{g \in X^{s,b}: g=f\ \text{on\ } [0,T]\}$, and define the quotient space whose norm is defined by $\| f \|_{X^{s,b}_T} = \inf\limits_{g \in \mathcal{C}_f} \| g \|_{X^{s,b}}$.

\begin{proposition}[{\cite[Theorem 1.1]{Yonggeun}}]\label{lwp_fnls}
    Given $\alpha \in (1,2)$ and $s \geq \frac{2-\alpha}{4}$, the fNLS
    \begin{equation}\label{fnls}
i \partial_t u = (-\Delta)^{\frac{\alpha}{2}} u + \mu |u|^2 u,\ (x,t) \in \mathbb{T} \times \mathbb{R}
\end{equation}    
    is locally well-posed in $H^s(\mathbb{T})$. More precisely, for any initial datum $u_0 \in H^s(\mathbb{T})$, there exists a unique $u \in X^{s,\frac{1}{2}+\epsilon}_T \subseteq C([0,T];H^s(\mathbb{T}))$, for every $0 < \epsilon \ll 1$, such that the integral representation of \eqref{fnls} given by
    \begin{equation*}
        u(t) = U(t) u_0 - i\mu \int_0^t U(t-\tau)\left(|u(\tau)|^2 u(\tau)\right) d\tau
    \end{equation*}
    holds for all $t \in [0,T]$ where $T = T(\| u_0 \|_{H^s})>0$. Furthermore mass ($M$) and energy ($H$) are conserved, where
    \begin{equation*}
        M[u(t)] = \| u(t) \|_{L^2}^2;\ H[u(t)] = \frac{1}{2} \| |\nabla|^{\frac{\alpha}{2}}u\|_{L^2}^2 + \frac{\mu}{4} \| u \|_{L^4}^4.
    \end{equation*}
\end{proposition}

A crucial estimate used in the proof of \Cref{lwp_fnls} is the following bilinear estimate.

\begin{proposition}[{\cite[Proposition 3.2]{Yonggeun}}]\label{bilinear}
For $s \geq \frac{2-\alpha}{4}$ and $0 < \epsilon \ll 1$, we have
\begin{equation*}
    \| uv \|_{L^2(\mathbb{R} \times \mathbb{T})} \lesssim_\epsilon \| u \|_{X^{0,\frac{1}{2}-\epsilon}} \| v \|_{X^{s,\frac{1}{2}+\epsilon}}.
\end{equation*}
\end{proposition}

The time interval $[0,T]$ of the local well-posedness of \eqref{fdnls} in the Strichartz space $X_T := C([0,T];L^2_h) \cap L^6([0,T];L^\infty_h)$ cannot be determined uniformly in $h>0$ solely from $L^p_h \hookrightarrow L^q_h$, for $p<q$, since the embedding is not uniform in $h$ as $\| f \|_{L^q_h} \leq h^{\frac{1}{p} - \frac{1}{q}} \| f \|_{L^p_h}$. An application of \Cref{strichartz} yields a uniform estimate in $h$.

\begin{proposition}\label{lwp_fdnls}
    Let $s > \frac{1}{3}$ and $\alpha \in (1,2]$. For every $u_{h,0} \in H^s_h$, there exists a unique $u_h \in X_T$ such that
    \begin{equation}\label{duhamel_discrete}
    u_h(t) = U_h(t) u_{h,0} - i\mu \int_0^t U_h(t-\tau)\left(|u_h(\tau)|^2 u_h(\tau)\right) d\tau    
    \end{equation}
    where $T_h \sim_{\alpha} \| u_{h,0}\|_{H^s_h}^{-3}$ is independent of $h>0$. Furthermore discrete mass ($M_h$) and discrete energy ($H_h$) are conserved, where
    \begin{equation*}
        M_h[u_h(t)] = \| u_h(t) \|_{L^2_h}^2;\ H_h[u_h(t)] = \frac{1}{2} \| |\nabla_h|^{\frac{\alpha}{2}}u_h\|_{L^2_h}^2 + \frac{\mu}{4} \| u_h \|_{L^4_h}^4.
    \end{equation*}
\end{proposition}

\begin{proof}
Let $\Gamma u_h$ be the RHS of \eqref{duhamel_discrete}. By \Cref{strichartz} and the a priori estimate,
\begin{equation*}
\begin{split}
\| \Gamma u_h \|_{X_T} &:= \| \Gamma u_h \|_{C_T H^s_h} + \| \Gamma u_h \|_{L^6_T L^\infty_h} \lesssim \| u_{h,0} \|_{H^s_h} + \| |u_h|^2 u_h \|_{L^1_T H^s_h}\\
&\lesssim \| u_{h,0} \|_{H^s_h} + \left|\left|\| u_h \|_{L^\infty_h}^2 \| u_h \|_{H^s_h}\right|\right|_{L^1_T} \leq \| u_{h,0} \|_{H^s_h} + T^{\frac{2}{3}} \| u_h \|_{L^6_T L^\infty_h}^2 \| u_h \|_{C_T H^s_h},\\
\| \Gamma u_h -  \Gamma v_h \|_{X_T} &\lesssim T^{\frac{2}{3}} (\| u_h \|_{X_T}^2 + \| v_h \|_{X_T}^2) \| u_h - v_h\|_{X_T},
\end{split}
\end{equation*}
there exists a unique fixed point $u_h$ in a small closed ball of $X_T$ such that $u_h = \Gamma u_h$. The domain of $\Gamma$ is extended from the neighborhood of the origin to the entire $X_T$ by the continuity argument.    
\end{proof}

\section{Appendix: exact solutions.}\label{appendix_exact}
To illustrate the trivial case, consider $u_0(x) = A \in \mathbb{C}$. By direct computation,
\begin{equation*}
    p_hS_h(t)d_h u_0(x) - S(t) u_0 (x) = A e^{-i\mu |A|^2 t} - A e^{-i\mu |A|^2 t} = 0.
\end{equation*}
Another example of exact solution is a family of sinusoids that oscillate at a single spatiotemporal frequency. Let $A \in \mathbb{C} \setminus \{0\},\ n \in \mathbb{Z} \setminus \{0\},\ s \in \mathbb{R}$, and consider $u_0(x) = A |n|^{-s}e^{inx}$ for $x \in \mathbb{T}$. By definition of $d_h$ and $\mathcal{F}_h$, and given the Fourier expansion $f(x) = \frac{1}{2\pi} \sum\limits_{k \in \mathbb{Z}}\widehat{f}(k)e^{ikx}$, we have 
\begin{equation}\label{dft}
d_h [e^{ik\cdot}](x^\prime) = \begin{cases}
			\frac{e^{ihk} - 1}{ihk} e^{ikx^\prime}, & k \neq 0\\
            1, & k = 0,
\end{cases}
;\ \mathcal{F}_h [d_h f] (k^\prime) = \begin{cases}
    \sum\limits_{q \in \mathbb{Z}} \frac{\widehat{f}(pq + k^\prime)}{pq + k^\prime} \frac{e^{ihk^\prime} - 1}{ih}, & k^\prime \neq 0\\
    \widehat{f}(0), & k^\prime = 0,
    \end{cases}
\end{equation}
where $x^\prime \in \mathbb{T}_h,\ k^\prime \in \mathbb{T}_h^*$. Note that the domain of $\mathcal{F}_h [d_h f]$ can be extended from $\mathbb{T}_h^*$ to $\mathbb{Z}$ periodically since the summation over $q \in \mathbb{Z}$ is over all periods with the period $2M = \frac{2\pi}{h}$. By direct computation,
\begin{equation}\label{exactsolution}
\begin{split}
S(t)u_0(x) &= A |n|^{-s} e^{-it(|n|^\alpha + \mu |A|^2 |n|^{-2s}} e^{inx},\ x \in \mathbb{T}\\
S_h(t) d_h u_0(x) &= A |n|^{-s}\frac{e^{ihn}-1}{ihn} e^{-it(|\frac{2}{h}\sin \frac{hn}{2}|^\alpha + \mu |A|^2 |n|^{-2s}|\frac{e^{ihn}-1}{ihn}|^2} e^{inx},\ x \in \mathbb{T}_h,
\end{split}
\end{equation}
i.e., the exact solutions to \eqref{fnls} and \eqref{fdnls}, respectively. Recalling 
\begin{equation*}
\mathcal{F}[p_h f](k) = P_h(k) \mathcal{F}_h [f] (k) = \left(\frac{\sin(\frac{hk}{2})}{\frac{hk}{2}}\right)^2 \mathcal{F}_h [f] (k),    
\end{equation*}
for $k \in \mathbb{Z}$, we have
\begin{align}\label{exact_cts}
\| p_h u_h(t) &- u(t) \|_{L^2(\mathbb{T})} = (2\pi)^{-\frac{1}{2}} \|  P_h(k) \mathcal{F}_h[u_h(t)](k) - \widehat{u(t)}(k)\|_{l^2_k}\nonumber\\
&= (2\pi)^{\frac{1}{2}} |A| |n|^{-s} \left| P_h(n) \frac{e^{ihn}-1}{ihn} e^{-it(|\frac{2}{h}\sin \frac{hn}{2}|^\alpha + \mu |A|^2 |n|^{-2s}|\frac{e^{ihn}-1}{ihn}|^2} - e^{-it(|n|^\alpha + \mu |A|^2 |n|^{-2s}} \right|\nonumber\\
&= \left(\sqrt{\frac{\pi}{2}}|A||n|^{1-s}\right)h + O(h^2),
\end{align}
which yields sharp linear convergence, where the last equality is by the Taylor's Theorem. 
\bibliographystyle{plain}
\small
\bibliography{ref.bib}

\end{document}